\documentclass[10pt]{article}
\usepackage{geometry}
\usepackage{graphicx}
\usepackage{amsmath,amssymb,amsthm}
\usepackage[utf8]{inputenc}
\usepackage[greek,english]{babel}
\usepackage{enumerate}
\usepackage{hyperref}
\usepackage[active]{srcltx}
\numberwithin{equation}{section}
\usepackage[T1]{fontenc}
\usepackage{color}
\newtheorem{theo}{Theorem}
\newtheorem{lem}{Lemma}[section]

\newtheorem{prop}{Proposition}[section]
\newtheorem{rmk}{Remark}[section]

\newcommand{\R}{\mathbb{R}}

\date{}

\begin{document}
\title{Monostable pulled fronts and logarithmic drifts}
\author{Thomas Giletti \footnote{IECL, University of Lorraine, UMR 7502, B.P. 70239, 54506 Vandoeuvre-l\`{e}s-Nancy Cedex, France}}

\maketitle

\abstract{In this work we investigate the issue of logarithmic drifts in the position of the level sets of solutions of monostable reaction-diffusion equations, with respect to the traveling front with minimal speed. On the one hand, it is a celebrated result of Bramson that such a logarithmic drift occurs when the reaction is of the KPP (or sublinear) type. On the other hand, it is also known that this drift phenomenon disappears when the minimal front speed is nonlinearly determined. However, some monostable reaction-diffusion equations fall in neither of those cases and our aim is to fill that gap. We prove that a logarithmic drift always occurs when the speed is linearly determined, but surprisingly we find that the factor in front of the logarithmic term may be different from the KPP case.}

\section{Introduction}

In this paper, we consider a scalar reaction-diffusion equation
\begin{equation}\label{eq:main}
\partial_t u = \partial_{xx} u + f(u), \quad t >0, \ x \in \R ,
\end{equation}
where $f \in C^2$ is of the monostable type, i.e.
\begin{equation}\label{eq:monostable}
f (0 ) = f(1) = 0 \ , \quad f'(0) > 0 > f'(1), \quad  \mbox{ and } \quad f (s) >0 \ \mbox{ for $s \in (0,1)$. }
\end{equation}
Equation~\eqref{eq:main} will be supplemented together with the initial condition
\begin{equation}\label{eq:ini0}
u (t= 0, \cdot) \equiv u_0 \in C^0 (\R ; [0,1]),
\end{equation}
where $u_0$ satisfies that there exists $X_0 >0$ such that
\begin{equation}\label{eq:ini}
\inf_{x \leq - X_0} u_0 (x) >0 = \sup_{x \geq X_0 } u_0 (x) .
\end{equation}
Notice that, by a comparison principle, there holds that $u(t,x) \in (0,1)$ for all $t >0$ and $x \in \R$, so that it can be assumed without loss of generality that
$$\| f' \|_\infty  + \| f'' \|_\infty < + \infty.$$

Such an equation is well-known to admit a family of traveling waves~\cite{AW75}, whose minimal speed $c^*$ satisfies
$$c^* \geq 2 \sqrt{f'(0)},$$
and is also the spreading speed of solutions of \eqref{eq:main} with initial data~\eqref{eq:ini0}-\eqref{eq:ini}; we will recall below what is meant by spreading speed. In particular, equation~\eqref{eq:main} is often used as a model for invasion phenomena in physics, ecology and population dynamics.

The constant $2 \sqrt{f'(0)}$, or linear speed, arises from replacing $f(u)$ in \eqref{eq:main} by $f'(0) u$ its linearization at 0. Under the additional assumption that
\begin{equation}\label{eq:kpp}
u \mapsto \frac{f (u)}{u} \ \mbox{ is decreasing for } \ u >0, 
\end{equation}
which is usually refered to as the Fisher or KPP case, then it is actually known that $c^* = 2 \sqrt{f '(0)}$~\cite{AW75}. However, the converse is not true and, without assumption~\eqref{eq:kpp}, this equality may or may not hold.

A typical example is
\begin{equation}\label{eq:example}f_a (u) = 	u (1-u) (1 + au),\end{equation}
where $a \geq 0$. Notice that the KPP assumption~\eqref{eq:kpp} is satisfied if $a \leq 1$. On the other hand, the minimal wave speed is given by the following formula~\cite{HR75}:
$$ c_a^* 
= 
\left\{
\begin{array}{ll}
2  & \mbox{ if $0 \leq a \leq 2$}, \vspace{3pt}\\
\sqrt{ \frac{2}{a} } + \sqrt{\frac{a}{2} } & \mbox{ if a > 2.}
\end{array}
\right.
$$
In particular, in the interval $a \in (1,2]$, the minimal wave speed $c^*$ is equal to the linear speed $2 \sqrt{f'(0)}$ even though condition~\eqref{eq:kpp} does not hold. We highlight the fact that this interval is nontrivial, which suggests that this situation is not merely theoretical but may indeed occur in the applications. \\

Furthermore, as mentioned above, it is well-known~\cite{AW75} that $c^*$ is also the spreading speed of solutions of \eqref{eq:main} under initial condition~\eqref{eq:ini0}-\eqref{eq:ini}. More precisely, any such solution satisfies that:
$$\forall c < c^*, \quad \lim_{t \to +\infty} \sup_{x \leq ct } |u (t,x) - 1 | = 0,$$
$$\forall c > c^*, \quad \lim_{t \to +\infty} \sup_{x \geq ct } |u (t,x) | = 0.$$
In particular, any level set between 0 and 1 must be located around $c^* t$, up to some~$o(t)$ as $t \to +\infty$. 

Yet one may want to describe the large time behavior of solutions more precisely, whether by estimating more precisely the position of level sets or by investigating the convergence of the profile of the solution to that of a traveling wave. It turns out that the picture differs depending on whether $c^* > 2 \sqrt{ f'(0)}$ (i.e. \textit{`pushed'} case) or~$c^* = 2 \sqrt{ f'(0)}$ (i.e. \textit{`pulled'} case). We refer to~\cite{Stokes} for a background on the pushed/pulled terminology.

In the pushed case when $c^* > 2 \sqrt{ f'(0)}$, the solution converges to a single shift of the traveling wave $U_*$ with minimal speed~\cite{Rothe81,Uchiyama}, i.e. there exists some $X_\infty \in \R$ (depending on the initial data) such that
$$ u (t,x) - U_*(x - c^* t + X_\infty) \to 0,$$
as $t \to +\infty$, where the convergence is uniform with respect to $x \in \R$. This simultaneously answers both questions of level set position and profile convergence.

However, in the case when $c^* = 2 \sqrt{f'(0)}$ and \eqref{eq:kpp} is satisfied, then some logarithmic drift occurs and the solution no longer persists in the moving frame with speed $c^*$. More precisely, 
$$ u(t,x) - U_*\left( x-c^* t + \frac{3}{2 \sqrt{f'(0)} } \ln t + X_\infty \right) \to 0,$$
as $t \to +\infty$, for some $X_\infty \in \R$. This has first been proved by a probabilistic approach in~\cite{Bramson}, and more recently with PDE tools in~\cite{HNRR}. Let us also mention recent developments in the periodic heterogeneous~\cite{HNRR_per} and nonlocal~\cite{bouin,Graham} cases, as well as for the Burgers-Fisher equation~\cite{JingAn,LeachHanac}.

While the above two results are now rather well-known, they leave aside the case when $c^* = 2 \sqrt{ f'(0)}$ yet \eqref{eq:kpp} does not hold. As we pointed out in the particular case when $f$ is given by \eqref{eq:example}, the set of reaction terms leading to this situation is not trivial and therefore it should also occur in the applications. It had only been shown~\cite{Rothe78,Uchiyama} that, for some special class of `steep enough' initial data, 
$$ u(t,x) - U_*\left( x-c^* t + m(t) \right) \to 0,$$
as $t \to +\infty$, where $m(t) = o(t)$ as $t \to +\infty$, and the logarithmic drift was (formally) studied in~\cite{Ebert,LeachNeedham}. More recently, local convergence and drift results were obtained in~\cite{AveryScheel}, under some spectral stability assumption on the traveling wave which is related to what we will call the slow decay case in the next sections.

Therefore, the goal of this paper is to investigate the drift phenomenon in the general pulled case $c^* = 2 \sqrt{f'(0)}$, i.e. without making the KPP assumption. We will see that a logarithmic drift still appears, but may involve a different factor.

\section{Main results}

Let us first recall that a traveling wave solution of \eqref{eq:main} is an entire in time solution of the form
$$u(t,x) = U (x-ct),$$
where $c \in \R$ and $U$ satisfies
\begin{equation}\label{TW_ineq}
U (-\infty) = 1> U (\cdot) > U (+\infty) = 0.
\end{equation}
As we outlined in the Introduction, such a traveling wave solution exists if and only if $c \geq c^*$, for some $c^* \geq 2 \sqrt{f'(0)}$. Moreover, for each $c \geq c^*$, the traveling wave solution is unique up to shift and the profile function $U$ is decreasing~\cite{AW75}. When $c= c^*$ we denote it by $U_*$.

Since the pushed case when $c^* > 2 \sqrt{f'(0)}$ is well-understood, throughout this work we will place ourselves in the pulled case where
\begin{equation}\label{eq:pulled}
c^* = 2 \sqrt{f'(0)} .
\end{equation}
We also recall that $U_* (z)$ satisfies the following asymptotics as $z \to +\infty$:
\begin{equation}\label{eq:asymp}
U_* (z) =   (B z + A) e^{- \sqrt{f'(0)} z} + O(e^{-(2-\eta) \sqrt{f'(0)} z}),
\end{equation}
where $\eta >0$ can be chosen arbitrarily small thanks to the $C^2$-regularity of~$f$~\cite{CoddingtonLevinson,Pazy}. 
We point out that in our arguments we will fully use~\eqref{eq:asymp} and in particular the order of the remainder. Still we expect that our main results should hold true if we only assume~$f$ to be $C^{1,\alpha}$, up to some appropriate changes in our proofs.

Moreover, we have that either $B>0$, or $B=0$ and $A >0$. We refer to the former case as \textit{`slow decay'} and to the latter as \textit{`fast decay'}. As predicted in~\cite{Ebert,LeachNeedham}, the logarithmic drift will be different in both cases. We also mention the recent work~\cite{BouinHenderson}, where Bouin and Henderson investigated a situation where the reaction term is singular, which results in different asymptotics of the minimal traveling wave and the logarithmic drift is then increased compared to the classical KPP case.\\

We are now in a position to state our main result, where the above asymptotics play a crucial role:
\begin{theo}\label{th:main}
Assume that $u$ solves \eqref{eq:main} with initial condition~\eqref{eq:ini0}-\eqref{eq:ini}, $f$ satisfies \eqref{eq:monostable} and~\eqref{eq:pulled} holds. Define also $B$ from \eqref{eq:asymp}, $\lambda \in (0,1)$, and the level set
$$E_\lambda (t) := \{ x \, | \ u(t,x) = \lambda \}.$$
\begin{itemize}
\item[$(i)$] If $B >0$, then there exist $X >0$ and $T>0$ such that
$$E_\lambda (t) \subset \left[  c^* t - \frac{3}{2 \sqrt{f'(0)} } \ln t - X , c^* t - \frac{3}{2 \sqrt{f'(0)}  } \ln t + X \right],$$
for all $t > T$.
\item[$(ii)$] If $B = 0$, then for any $\varepsilon >0$, there exists $T>0$ such that
$$E_\lambda (t) \subset \left[  c^* t - \frac{1+ \varepsilon}{2 \sqrt{f'(0)} } \ln t  , c^* t - \frac{1-\varepsilon}{2 \sqrt{f'(0)} } \ln t \right],$$
for all $t > T$.
\end{itemize}
\end{theo}
Statement~$(i)$ includes the KPP case and we recover the same logarithmic drift; we again cite the parallel work~\cite{AveryScheel} where a similar result was proved by a spectral approach, which unfortunately fails in the fast decay case. This is the most typical situation in the sense that it is stable up to small $C^1$-perturbations of the reaction term. For instance, in the example of $f_a (u) = u (1-u) (1+au)$, there holds $B>0$ for any $a < 2$ (recall that the KPP assumption is satisfied only when $a \leq1$); this can be checked by phase plane analysis. In the special case when $B=0$, we find some new phenomenon where the logarithmic drift still appears but is less than in the KPP case. It is straightforward to check that this new situation does occur, for instance by taking $a=2$ and the reaction term $f_2 (u) = u (1-u) (1+2u)$, and noticing that the traveling wave with minimal speed is then given by $U_* (z) = (1+e^z)^{-1}$ whose decay is fast. However, in this example there is no other value of the parameter $a$ which falls into this fast decay case. Recalling that the speed is nonlinearly determined for $a >2$, this suggests that the fast decay case appears only as a sharp intermediate situation between a KPP-like logarithmic drift and no drift.

Let us briefly point out that, in the KPP case, the correct logarithmic drift is obtained by a truncation procedure in which the Cauchy problem \eqref{eq:main} is approached by the linearized problem but with a moving Dirichlet boundary condition. More precisely, it is approximated by the equation
$$\partial_t u = \partial_{xx} u + f' (0) u,$$
together with
$$u (t, 2t - r \ln t ) = 0.$$
The resulting solution remains bounded from both above and below in large time if and only if $r = \frac{3}{2}$, which turns out to be the correct logarithmic drift in this case. The Dirichlet boundary condition basically comes from the fact that it is satisfied by (a shift of) the function $(B z + A) e^{- \sqrt{f'(0)} z }$ (when $B >0$), which is also the asymptotic profile of the traveling wave with minimal speed. In the case when $B= 0$ and the front has the faster decay $Ae^{-\sqrt{f'(0)} z}$ as $z \to -\infty$, a different boundary condition should be imposed and this is, in short, the reason why a different logarithmic drift occurs.

In the slow decay case, our estimate on the position of level sets is sharp, and proceeding as outlined in~\cite{HNRR}, one may recover the large time convergence of the solution to a family of shifted traveling waves with minimal speed. Unfortunately, in the fast decay case we were only able to locate the level sets up to the order $o (\ln t)$. Very roughly, this seems to be related to the fact that we have one less parameter to play with in the asymptotic profile $A e^{-\sqrt{f'(0)} z}$ of the traveling wave, when matching it with the approximated problem sketched above. Still it is expected that our estimate should be valid up to a~$O(1)$, and at least the distance between level sets should remain at a bounded distance from each other as $t \to +\infty$.

\paragraph{Plan of the paper.} In order to exhibit the new phenomenon and give some of the key ideas, we start in Section~\ref{sec:linear} with a special case when the front has fast decay under the additional assumption that $f$ is linear around 0, i.e. $f(u) = f'(0) u$ in an open neighborhood of 0. In this situation, a short argument provides a lower estimate on the position of level sets, and shows that the drift cannot be the same as in the KPP case.

In the last two sections, we turn to the actual proof of our main Theorem~\ref{th:main} in the general pulled case. First, in Section~\ref{sec:subsuper}, we construct some sub and supersolutions in both the fast decay and the slow decay cases, and whose level sets satisfy the wanted asymptotics. Then, in Section~\ref{sec:comparison} and using these sub and supersolutions, we conclude the proof of Theorem~\ref{th:main}.

Before we proceed, we point out that, in order to make the computations simpler and without loss of generality (up to some rescaling), all the proofs will be performed under the additional assumption that
$$f '(0) = 1,$$
so that also $c^* = 2$ and
$$U_* '' + 2 U_* ' + f(U_*)=0.$$

\section{A rough argument in a simple case}\label{sec:linear}

For the sake of illustrating the quite intricate argument used in the sequel, we briefly discuss a special case when, for some $\delta >0$,
\begin{equation}\label{eq:linear}
\forall s \in [0, \delta], \quad f(s) = f'(0) s = s . 
\end{equation}
Our point is also to provide an example where the lower bound on the level set in statement~$(ii)$ of Theorem~\ref{th:main} can be verified, in a more straightforward way which avoids the technicalities of the following sections.


Here we place ourselves in the fast decay case, which under condition~\eqref{eq:linear} implies that
$$U_* (z) = A e^{-z},$$
for some $A >0$ on a right half-line. Up to some shift, we can assume without loss of generality that
\begin{equation}\label{TW_shift_norm}
U_*(z) = \delta e^{-z}, \quad \forall z \geq 0 .
\end{equation}
We will sketch a short argument to show that in this case the logarithmic drift is at most $\frac{1}{2} \ln t$. This already highlights a difference between the general monostable case and the particular KPP case studied in~\cite{HNRR}. 

More precisely, we present a short construction of a subsolution. This construction is made easier by the linear assumption and the resulting exact asymptotics for $U_*$, which allow us to compare the steepness of various solutions. A lower estimate on the position of the level sets follows from a comparison principle argument; here we omit this part of the proof, since it is the same as in the general case tackled in later sections.\\

Let us start with a formal argument. When the solution is less than $\delta$ from~\eqref{eq:linear}, it satisfies the linear equation
$$\partial_t u= \partial_{xx} u + u,$$
whose fundamental solution writes
$$U_0 (t,x) :=  \delta \frac{e^t}{\sqrt{t}} \times e^{-\frac{x^2}{4t}}.$$
Here we added the factor $\delta$ for convenience. Notice indeed that $U_0 (t,x) \leq \delta$ if and only if
$$x \geq 2t -h (t),$$
where
$$h(t) = 2 t - \sqrt{ 4 t^2 - 2 t \ln (t)}.$$
In particular, as $t \to +\infty$,
$$h(t) \sim \frac{1}{2} \ln (t) + o(1).$$
On the left of the point $2 t - h(t)$, the function~$U_0$ is larger than $\delta$ and we instead expect the solution to approach a shift of the traveling front with minimal speed $U_*$. In order to ensure continuity at the point $x=  2t - h(t)$ and due to \eqref{TW_shift_norm}, more precisely we expect the solution to approach
$$U_* (x - 2t + h(t)),$$
so that $U_* (0) = U_0 (t, 2t -h(t))$. From the above formula for $h(t)$ this provides the expected logarithmic drift on the position of level sets.

To make it more rigorous, we may check that the function
$$\underline{U} (t,x) := \left\{ 
\begin{array}{ll}
U_*(x - 2t + h(t)) & \mbox{ if } x \leq 2t - h(t) , \vspace{3pt}\\
U_0 (t,x) & \mbox{ if } x > 2 t - h(t).
\end{array}
\right.
$$
is a (generalized) subsolution of~\eqref{eq:main} for positive times.

Indeed, by construction, $U_0 (t,x)$ is a solution of \eqref{eq:main} for $t >0$ and $x > 2t - h(t)$, where it is less than $\delta$. The fact that $U_* (x -2t + h(t))$ is a subsolution of \eqref{eq:main} easily follows from the fact that $U_* ' < 0$ and  $h ' (t) > 0$ for all $t >1$. Therefore, the function~$\underline{U}$ is a generalized subsolution for $t >1$ if
$$U_*'  (0) < \partial_x U_0 (t, 2t -h(t)).$$
The left-hand term is equal to $-\delta$, and the right-hand term is given by
$$\delta \frac{e^{t}}{\sqrt{t}} \times \frac{-(2t -h(t))}{2t} e^{-\frac{(2t - h(t))^2}{4t} }= - \delta + \frac{ h(t)}{2t} \delta > - \delta, $$
where the last inequality holds for $t>1$. 

It follows that, as announced, $\underline{U}$ is a subsolution. Although it is positive on the whole real line, it has the same decay as solutions of \eqref{eq:main} with initial data satisfying~\eqref{eq:ini}; hence it is natural to expect that a comparison principle can be applied. Furthermore, the level sets of $\underline{U}$ are located around $2t - h(t)$, which is precisely the expected logarithmic drift in the fast decay case~$B=0$.

\section{The general monostable case: sub and supersolutions}\label{sec:subsuper}

We now turn to the construction of sub and supersolutions in the general case when $f$ satisfies~\eqref{eq:monostable}.

\subsection{The fast decay case: a subsolution}\label{sec:fast_sub}

In this section we assume that $U_*$ has fast decay in the sense that
$$U_* (z) = e^{-z} + O (e^{-(2-\eta) z}),$$
as $z \to +\infty$, for any $\eta >0$. Notice that we made the constant $A$ equal to 1, which is possible up to some spatial shift.

Here we will let
$$r > \frac{1}{2},$$
to be chosen arbitrarily close to $\frac{1}{2}$. 
Our goal is to prove the following:
\begin{prop}\label{fast_subsol}
For any $r \in \left( \frac{1}{2}, 1\right)$, there exist $T>0$ and $\underline{u} (t,x)$ a subsolution of~\eqref{eq:main} on $[T,+\infty) \times \R$ such that:
\begin{itemize}
\item[$(i)$] for any $\lambda \in (0,1)$, there exists $X_1 >0$ such that, for any $t$ large enough,
$$\forall x < 2 t - r \ln t - X_1, \quad \underline{u} (t,x) > \lambda,$$
$$\forall x > 2 t - r \ln t + X_1, \quad \underline{u} (t,x) < \lambda ;$$
\item[$(ii)$] it satisfies the following properties at time $T$:
$$\sup_{x \in \mathbb{R}} \underline{u} (T,x) < 1 , $$
$$\exists X_2 >0, \quad \limsup_{x \to +\infty} \frac{(x- X_2) \underline{u} (T,x)}{e^{-(x- X_2)^2 /4T}} =0.$$
\end{itemize}
\end{prop}
The construction of this subsolution will consist in gluing together two different functions, roughly dealing respectively with the nonlinear part (away from 0) and the linear part (close to 0) of \eqref{eq:main}. First, our subsolution will consist on a left half-line of a (perturbation of a) shift of the traveling front with minimal speed, so that its level sets are located around $2t -  r \ln t$. Indeed, one may check that
$$U_*(x - 2t + r \ln t - h(t) ) ,$$
is a subsolution of \eqref{eq:main} at large times if 
\begin{equation}\label{fast_subsol1}
h(t) =O(1), \quad  h'(t) = o \left( \frac{1}{t} \right) \ \mbox{ as } \ t \to +\infty,
\end{equation}
thanks to the fact that $U_* ' < 0$ and $U_* '' + 2 U_* ' + f(U_*) =0$.

Because $u(t,x)$ the solution of the Cauchy problem \eqref{eq:main} with initial data~\eqref{eq:ini0}-\eqref{eq:ini} has a faster decay than $U_*$ as $x \to +\infty$, it is not possible to use a comparison principle. This is why another subsolution will be necessary on the far right. Another issue is the fact that $U_*$ goes to 1 as $x \to - \infty$, which may not be the case of the solution $u(t,x)$. The latter difficulty can be solved thanks to the linear stability of 1 and by using a similar construction as in~\cite{FifeMcLeod}. This is the purpose of the following lemma.
\begin{lem}\label{lem:ansatz1}
Assume that $h(t) $ satisfies~\eqref{fast_subsol1}.

Then there exist a nontrivial and nonincreasing function~$\chi$, whose support is included in $(-\infty,0)$ and which satisfies $\chi (-\infty) =1$, and $T>0$ such that
$$\underline{u}_1 (t,x) := U_* (x - 2t + r \ln t - h(t) ) - \chi (x- 2t + r \ln t - h(t)) e^{ \frac{f'(1)}{2} t}, $$
is a subsolution of \eqref{eq:main} on $[T,+\infty) \times \R$.
\end{lem}
We point out that this lemma actually holds true whatever the choice of $r>0$. The choice of $r > \frac{1}{2}$ will only be needed in the later stages of the proof of Proposition~\ref{fast_subsol}. 
\begin{proof}
Consider $\delta >0$ small enough so that
\begin{equation}\label{delta_around1}
f ' (u) \leq \frac{f'(1)}{2} < 0,
\end{equation}
for any $u \in [1-\delta, 1]$. Now take $Z >0$ such that
$$U_* (z) \geq 1 - \frac{\delta}{2},$$
for all $z \leq -Z$, and $\chi$ a smooth and nonincreasing function which is identical to 1 on $(-\infty, -Z-1)$ and identical to 0 on $(-Z, +\infty)$. Define also $T >0$ such that, for all $t \geq T$,
\begin{equation}\label{fast_subsol1_eq1}
e^{\frac{f'(1)}{2} t} \leq \frac{\delta}{2}\ , \quad  1 \geq  \frac{r}{t} - h'( t) \geq \frac{r}{2t}.
\end{equation}
It is then straightforward to compute that
\begin{eqnarray*}
& & \partial_t \underline{u}_1 - \partial_{xx} \underline{u}_1 - f (\underline{u}_1) \\
& = &  \left( \frac{r}{t} - h' (t) \right) U_* ' - \left( U_* '' + 2 U_* '+ f(U_*) \right) + f(U_*) - f( \underline{u}_1)   \\
& & \quad +  \left(- \frac{f'(1)}{2} \chi +  \left(2 - \frac{r}{t} + h'(t) \right) \chi ' + \chi ''   \right) e^{ \frac{f'(1)}{2} t} \\
& \leq &  \left( \frac{r}{t} - h' (t) \right) U_* '  + f(U_*) - f( \underline{u}_1)   +  \left( -\frac{f'(1)}{2} \chi +  \chi ' + \chi ''   \right) e^{\frac{f'(1)}{2} t} ,
\end{eqnarray*}
where $U_*$ and $\chi$ are evaluated at $x - 2t + r \ln t - h(t) $. Here we used the facts that $\chi ' \leq 0$, and $2 - \frac{r}{t} + h'(t) \geq 1$ by \eqref{fast_subsol1_eq1}.

Recall that $U_*$ is decreasing. Moreover, for any $t \geq T $ and $x < 2 t - r \ln t + h(t) - Z-1$, we have that $\chi = 1$, $\chi ' = \chi '' = 0$. From our choice of $Z$ and $T$ it follows that $$U_* (x - 2t + r \ln t - h(t)) \geq \underline{u}_1 (t,x) \geq U_* (-Z -1) - e^{\frac{f'(1)}{2} t} \geq  1 -\delta .$$
Thus, using also~\eqref{delta_around1},
\begin{eqnarray*}
& & \partial_t \underline{u}_1 - \partial_{xx} \underline{u}_1 - f (\underline{u}_1) \\
& \leq &  f(U_*) - f( \underline{u}_1)    -   \frac{f'(1)}{2} e^{\frac{f'(1)}{2} t} \\
& \leq & \frac{f'(1)}{2} e^{\frac{f'(1)}{2} t} -   \frac{f'(1)}{2} e^{ \frac{f'(1)}{2} t} \\
& \leq & 0.
\end{eqnarray*}
On the other hand, for $t \geq T$ and $x > 2 t - r \ln t + h(t) - Z$, we have that $\chi = \chi ' = \chi '' =0$, hence
\begin{equation*}
 \partial_t \underline{u}_1 - \partial_{xx} \underline{u}_1 - f (\underline{u}_1)  \leq   \left( \frac{r}{t} - h' (t) \right) U_* '  \leq 0 .
\end{equation*}
Finally, for $t \geq T$ and $x - 2 t + r \ln t - h(t) \in [ - Z- 1, -Z]$, we have from~\eqref{fast_subsol1_eq1} that
\begin{eqnarray*}
& & \partial_t \underline{u}_1 - \partial_{xx} \underline{u}_1 - f (\underline{u}_1) \\
& \leq &  \frac{r}{2t} U_* '  +    \left( \| f' \|_\infty \| \chi\|_\infty - \frac{f'(1)}{2} \| \chi \|_\infty +  \| \chi '  \|_\infty+ \| \chi ''  \|_\infty \right) e^{\frac{f'(1)}{2} t} .
\end{eqnarray*}
Since $\max_{[-Z-1 , -Z]} U_* ' < 0$, we get that this is nonpositive up to enlarging $T$. The lemma is proved.
\end{proof}

Let us now turn to the construction of the second subsolution, which shares some similarities with the arguments in~\cite{HNRR}. Here we will look at the equation \eqref{eq:main} in a moving frame, by letting $z = x  - 2t + \frac{1}{2} \ln t$, that is
\begin{equation}\label{eq:main_moving}
\partial_t u = \partial_{zz}  u + \left( 2 - \frac{1}{2t} \right)  \partial_z u + f(u).
\end{equation}
For now we neglect the nonlinear part and look at the resulting linearized equation
\begin{equation}\label{eq:main_moving_linear}
\partial_t u = \partial_{zz}  u + \left( 2 - \frac{1}{2t} \right)  \partial_z u + u.
\end{equation}
We will pick a particular subsolution on a right half-line $\{ z \geq 0\}$. To do so, we again rewrite the equation. We let $u = e^{-z} v$ and find
$$\partial_t v= \partial_{zz} v + \frac{1}{2t} ( v  - \partial_z v),$$
and in the self-similar variables $\tau = \ln t$ and $y = \frac{z}{\sqrt{t}}$,
$$\partial_\tau v = \partial_{yy} v +  \frac{y}{2} \partial_y v + \frac{1}{2} v - \frac{1}{2} e^{- \frac{\tau}{2}} \partial_y v .$$
As in~\cite{HNRR}, we then let
$$v = w \times e^{-\frac{y^2}{8}}  e^{- \frac{\tau}{2}} ,$$
and get
\begin{equation}\label{eq:nonautonomous}
\partial_\tau w = \partial_{yy} w - \left(\frac{y^2}{16} -\frac{3}{4} \right)w  - \frac{1}{2} e^{-\frac{\tau}{2}} \left(\partial_y w - \frac{y}{4} w\right).
\end{equation}
Notice that the autonomous part,
\begin{equation*}
\partial_\tau w = \partial_{yy} w - \left(\frac{y^2}{16} -\frac{3}{4} \right)w ,
\end{equation*}
admits a family of particular solutions
$$ C_1 ye^{- \frac{y^2}{8}} + C_2 e^{-\frac{y^2}{8}} e^{\frac{\tau}{2}},$$
with $C_1, C_2 \in \R$. This comes from a combination of the Dirichlet and Neumann first eigenfunctions, respectively $ye^{-y^2/8}$ and $e^{-y^2/8}$. More precisely these are positive eigenfunctions of the elliptic operator $\mathcal{L}w := \partial_{yy}w - \left(\frac{y^2}{16} - \frac{3}{4} \right) w$ on the positive half-line $(0,+\infty)$ with either Dirichlet or Neumann boundary conditions at 0. Notice that the Neumann eigenfunction plays the dominant role here as $\tau \to +\infty$, while in~\cite{HNRR} it was the Dirichlet eigenfunction. This choice is of course consistent with the fast decay of the traveling wave $U_*$ at $+\infty$, and it will provide the correct logarithmic drift in this case. The reason we include the Dirichlet eigenfunction in this combination is that it increases the slope at $y=0$ of the resulting function, which will eventually help us merging this with $\underline{u}_1$ into a subsolution of~\eqref{eq:main}.

Before we proceed, let us replace $C_1$ and $C_2$ by some well-chosen functions of time in order to deal with the nonautonomous part of~\eqref{eq:nonautonomous} (as well as the nonlinear part of~\eqref{eq:main_moving}, as we will compute below) and obtain a subsolution. More precisely, we define
\begin{equation}\label{eq:defw_subsol0}
\underline{w} (\tau, y) := (1+ C e^{-\frac{\tau}{2}}) \times \left[ ye^{-\frac{y^2}{8}}  + e^{-\frac{y^2}{8}} e^{\tau \left(\frac{1}{2}-\varepsilon\right)}\right] .
\end{equation}
with $C>0$ and $\varepsilon \in (0,1/2)$ to be specified later. Notice that the exponential growth in time of the Neumann component is slowed down by the inclusion of this parameter~$\varepsilon$. This is consistent with our aim of constructing a subsolution, but also it makes the Dirichlet component relatively larger. This will be crucial later when merging our supersolutions. For convenience, we also introduce
$$w_0 (\tau, y) := ye^{-\frac{y^2}{8}} + e^{-\frac{y^2}{8}} e^{\tau \left(\frac{1}{2}-\varepsilon\right)} ,$$
so that $\underline{w} (\tau, y) = (1 + C e^{-\frac{\tau}{2}}) w_0 (\tau,y)$ and 
\begin{equation*}
\partial_\tau w_0 = -\varepsilon   e^{-\frac{y^2}{8}} e^{\tau \left(\frac{1}{2}-\varepsilon\right)} +  \partial_{yy} w_0 - \left(\frac{y^2}{16} -\frac{3}{4} \right)w_0 .
\end{equation*}
Then we have that 
$$\partial_\tau \underline{w} = -\varepsilon ( 1 + C^{-\frac{\tau}{2}})   e^{-\frac{y^2}{8}} e^{\tau \left(\frac{1}{2}-\varepsilon\right)}  -\frac{C}{2} e^{-\frac{\tau}{2}} w_0 + \partial_{yy} \underline{w}  -\left(\frac{y^2}{16} -\frac{3}{4} \right)\underline{w}. $$
Thus
\begin{eqnarray*}
& & \partial_\tau \underline{w} - \partial_{yy} \underline{w} + \left(\frac{y^2}{16} -\frac{3}{4} \right)\underline{w}  + \frac{1}{2} e^{-\frac{\tau}{2}} \left(\partial_y \underline{w} - \frac{y}{4} \underline{w}\right) \\
& \leq &  \frac{e^{-\frac{\tau}{2}}}{2} \left( -C w_0 +  \partial_y \underline{w} - \frac{y}{4} \underline{w}   \right)\\
& \leq & \frac{e^{-\frac{\tau}{2}}}{2} e^{-\frac{y^2}{8}} \left( - C y - C e^{\tau \left( \frac{1}{2} - \varepsilon\right)}  +  (1 + C e^{-\frac{\tau}{2}}) \Big(1 -\frac{y^2}{2} - \frac{y}{2} e^{\tau \left( \frac{1}{2} - \varepsilon\right)}\Big)\right) \\
& \leq & \frac{e^{-\frac{\tau}{2}}}{2} e^{-\frac{y^2}{8}} \left( - C e^{\tau \left( \frac{1}{2} - \varepsilon\right)}  +  1 + C e^{-\frac{\tau}{2}} \right) \\
& \leq & 0,
\end{eqnarray*}
for all $y \geq 0$ and $\tau$ large enough. Going back to problem~\eqref{eq:main_moving_linear}, this gives some large time $T>0$ and
\begin{equation}\label{tildeu_new}
\tilde{u} (t,z) :=  \left( 1+ \frac{C}{\sqrt{t}} \right) \times  \left[  \frac{1}{t^{\varepsilon}} +   \frac{z}{t} \right] e^{-z} e^{- \frac{z^2}{4t}}.
\end{equation}
which satisfies 
$$\partial_t \tilde{u} \leq \partial_{zz} \tilde{u} + \left( 2 - \frac{1}{2t} \right) \partial_z \tilde{u} + \tilde{u},$$
for all $t \geq T$ and $z \geq 0$. Actually this also provides us with a subsolution of \eqref{eq:main}, as we show in the following lemma:
\begin{lem}\label{lem:ansatz2}
For any $\varepsilon \in \left(0,\frac{1}{2} \right)$, there exist $C>0$ and $T>0$ such that the function
\begin{align}
\underline{u}_2 (t,x) & :=   \left( 1+ \frac{C}{\sqrt{t}} \right) \times  \left[ \frac{1}{t^{\varepsilon}} +   \frac{x - 2t + \frac{1}{2} \ln t}{t} \right] e^{-(x - 2t + \frac{1}{2} \ln t)} e^{- \frac{(x - 2t + \frac{1}{2} \ln t)^2 }{4t}} , \notag
\end{align}
is a subsolution of \eqref{eq:main} in the subdomain $t \geq T$ and $x \geq 2 t + \ln t - 2 $.
\end{lem}
\begin{proof} 
Notice that $\overline{u}_2$ and $\tilde{u}$ from~\eqref{tildeu_new} are the same up to the change of variables $z = x - 2t + \frac{1}{2} \ln t $. In particular it is enough to prove that, up to increasing $T$, the function $\tilde{u}(t,z)$ also satisfies
$$\partial_t \tilde{u} \leq \partial_{zz} \tilde{u}  + \left( 2 - \frac{1}{2t} \right) \partial_z \tilde{u} + f (\tilde{u}),$$
for all $t \geq T$ and $z \geq \frac{3}{2} \ln  t -2$.
Indeed, we have
$$\partial_t \tilde{u} \times e^{z} e^{\frac{z^2}{4t}}= - \frac{C}{2 t^{3/2}} \left( \frac{1}{t^{\varepsilon}} + \frac{z}{t} \right) - \left( \frac{\varepsilon}{t^{1+ \varepsilon }} + \frac{z}{t^{2}} \right)  \left( 1 + \frac{C}{\sqrt{t}} \right) + \frac{z^2}{4t^2} \left( 1+ \frac{C}{\sqrt{t}}  \right)\left( \frac{1}{t^{\varepsilon}} + \frac{z}{t} \right).$$
Moreover,
\begin{eqnarray*}
\partial_z \tilde{u} \times e^{z} e^{\frac{z^2}{4t}} 
& = & \left(1 + \frac{C}{\sqrt{t}}  \right) \frac{1}{t} - \left( 1 + \frac{z}{2t} \right) \left(1 + \frac{C}{\sqrt{t}}  \right) \left( \frac{1}{t^{\varepsilon}}  + \frac{z}{t} \right) ,
\end{eqnarray*}
and
\begin{eqnarray*}
\partial_{zz} \tilde{u} \times e^{z} e^{\frac{z^2}{4t}}  &=& - \left( 1 + \frac{z}{2t} \right) \left(1 + \frac{C}{\sqrt{t}}  \right) \frac{1}{t} - \frac{1}{2t} \left(1 + \frac{C}{\sqrt{t}}  \right) \left( \frac{1}{t^{\varepsilon}}  + \frac{z}{t} \right)  \\
& & -  \left( 1 + \frac{z}{2t} \right) \left(1 + \frac{C}{\sqrt{t}}  \right) \frac{1}{t} +  \left( 1 + \frac{z}{2t} \right)^2 \left(1 + \frac{C}{\sqrt{t}}  \right)\left( \frac{1}{t^{\varepsilon}}  + \frac{z}{t} \right) \\
& = & \left(1 + \frac{C}{\sqrt{t}} \right) \left( - \frac{2}{t} - \frac{z}{t^{2}}  \right) + \left(1 + \frac{C}{\sqrt{t}} \right) \left( \frac{1}{t^{\varepsilon}}  + \frac{z}{t} \right) \left(-\frac{1}{2t} + 1 + \frac{z}{t} + \frac{z^2}{4t^2} \right) .
\end{eqnarray*}
Therefore
\begin{eqnarray*}
& & e^{z} e^{\frac{z^2}{4t}} \times \left[ \partial_t \tilde{u} - \partial_{zz} \tilde{u} - \left( 2 - \frac{1}{2t} \right) \partial_z \tilde{u}  - \tilde{u} \right]\\
& = &  \left(1 + \frac{C}{\sqrt{t}} \right) \left( \frac{1}{t^{\varepsilon}} + \frac{z}{t} \right) \left[  \frac{1}{2t} -2  - \frac{z}{t} + \left( 2 - \frac{1}{2t} \right) \left( 1 + \frac{z}{2t} \right) \right] \\
& & + \left(1 + \frac{C}{\sqrt{t}} \right) \left[  \frac{2}{t} - \frac{\varepsilon}{t^{1 + \varepsilon}} - \left( 2 - \frac{1}{2t} \right) \frac{1}{t}  \right] - \frac{C}{2 t^{3/2}} \left( \frac{1}{t^{\varepsilon}}  + \frac{z}{t} \right) \\
& = &  - \frac{z}{4t^2} \left(1 + \frac{C}{\sqrt{t}} \right) \left( \frac{1}{t^{\varepsilon}} + \frac{z}{t} \right)  + \left(1 + \frac{C}{\sqrt{t}} \right) \left[  \frac{1}{2t^2} - \frac{\varepsilon}{t^{1+\varepsilon}}  \right]   - \frac{C}{2 t^{3/2}} \left( \frac{1}{t^{\varepsilon}} + \frac{z}{t} \right) \\
& \leq & - \frac{C}{2t^{3/2}} \left( \frac{1}{t^\varepsilon} + \frac{z}{t} \right),
\end{eqnarray*}
where the last inequality holds for all $z \geq 0$ and $t \geq T$ (up to increasing $T$). As announced we have a subsolution of~\eqref{eq:main_moving_linear}. Furthermore, letting $K = \| f '' \|_\infty$, we then also have
\begin{eqnarray*}
& & \partial_t \tilde{u} - \partial_{zz} \tilde{u} - \left( 2 - \frac{1}{2t} \right) \partial_z \tilde{u}  - f(\tilde{u}) \\
& \leq  & \partial_t \tilde{u} - \partial_{zz} \tilde{u} - \left( 2 - \frac{1}{2t} \right) \partial_z \tilde{u}  -\tilde{u} + K \tilde{u}^2 \\
& \leq &  e^{-z} e^{-\frac{z^2}{4t}} \left( \frac{1}{t^{\varepsilon}} + \frac{z}{t} \right) \left[ - \frac{C}{2t^{3/2}} +  K \left( 1+ \frac{C}{\sqrt{t}} \right)^2 \left( \frac{1}{t^{\varepsilon}} +  \frac{z}{t} \right) e^{-z} e^{-\frac{z^2}{4t}} \right].
\end{eqnarray*}
Next, we choose $C$ large enough so that
\begin{equation}\label{choice_C}
C \geq 4 K e^2  \max_{ s \geq 0} ( 1 + s ) e^{- \frac{s^2}{4}}.
\end{equation}
and up to increasing~$T$ if needed, we assume that
$$\frac{3}{2} \ln t - 2 \geq 0, \quad 1 + \frac{C}{\sqrt{t}} < \sqrt{2},$$
for all $t \geq T$. In particular~$T>1$. Now, for $t \geq T$ and $z\geq  \frac{3}{2} \ln t -2$, it follows from~\eqref{choice_C} that
\begin{eqnarray*}
& & \partial_t \tilde{u} - \partial_{zz} \tilde{u} - \left( 2 - \frac{1}{2t} \right) \partial_z \tilde{u}  - f(\tilde{u}) \\
& \leq & e^{-z} e^{-\frac{z^2}{4t}} \left( \frac{1}{t^{\varepsilon}} + \frac{z}{t} \right) \left[  - \frac{C}{2t^{3/2}} + \frac{K}{t^{3/2}} \times 2 \left( \frac{1}{t^\varepsilon} +  \frac{z}{t} \right) e^{2-\frac{z^2}{4t}}   \right] \\
& \leq & e^{-z} e^{-\frac{z^2}{2t}} \left( \frac{1}{t^{\varepsilon}} + \frac{z}{t} \right) \frac{1}{2 t^{3/2}}  \left[  - C +  4K  \left( 1 +  \frac{z}{ \sqrt{t}} \right) e^{2-\frac{z^2}{4t}}   \right] \\
& \leq & 0.
\end{eqnarray*}
This ends the proof of Lemma~\ref{lem:ansatz2}.
\end{proof}

We are now in a position to conclude the proof of Proposition~\ref{fast_subsol}.
Let us glue together our two subsolutions $\underline{u}_1$ and $\underline{u}_2$. To do so, we choose
\begin{equation}\label{ht_fast1}
h(t) = \ln \left( 1 + \frac{C}{\sqrt{t}} \right) +  \ln \left( 1 + \frac{3\ln t}{2t^{1-\varepsilon}} \right) - \frac{9 (\ln t)^2}{16t},
\end{equation}
where $C$ comes from Lemma~\ref{lem:ansatz2}. In particular
$$e^{h(t)} = \left( 1 + \frac{C}{\sqrt{t}} \right) \left( 1 +  \frac{ 3 \ln t}{2 t^{1-\varepsilon}} \right)    e^{- \frac{9 (\ln t)^2}{16t}} $$
and also we have that 
$$h(t) \to 0 \ \mbox{ and } \ h'(t) = O (t^{-3/2}),$$
as $t \to +\infty$. In particular it satisfies \eqref{fast_subsol1} and Lemma~\ref{lem:ansatz1} applies. As we will see in the next computations, $h$ is chosen so that $\underline{u}_1 (t,x)$ intersects~$\underline{u}_2 (t ,x)$ around $2t + \ln t$, which according to Lemma~\ref{lem:ansatz2} is precisely where the latter function~$\underline{u}_2$ becomes a subsolution.

First, recall that
$$U_*(z) = e^{-z} + O (e^{-( 2 - \eta) z}) \ \mbox{ as } \ z \to +\infty,$$
with $\eta >0 $ arbitrarily small.

Then, recalling \eqref{fast_subsol1}, the definition of $\underline{u}_1$ in Lemma~\ref{lem:ansatz1} and also that $r > 1/2$, we get as $t \to +\infty$ that
$$\underline{u}_1 (t, 2t + \ln t -1 ) = U_*( (r +1) \ln t - 1 - h(t) ) = \frac{e^{1} e^{h(t)}}{t^{r +1}} + O \left(\frac{1}{t^{3}}\right).$$
On the other hand
\begin{eqnarray*}
\underline{u}_2 (t, 2 t + \ln t -1)& =&  \frac{1}{t^{\varepsilon}} \left( 1 + \frac{C}{\sqrt{t}} \right) \left( 1 +  \frac{ \frac{3}{2} \ln t -1}{t^{1-\varepsilon}} \right) \frac{e^{1}}{t^{3/2}}    e^{- \frac{(\frac{3}{2} \ln t - 1)^2 }{4t}} \\
& = &    \frac{e^{1} e^{h(t)} }{t^{\varepsilon+ 3/2}}    e^{ \frac{3 \ln t}{4t}  - \frac{1}{ 4t}}  -  \left( 1 + \frac{C}{\sqrt{t}} \right) \frac{e^{1}}{t^{5/2}}e^{- \frac{(\frac{3}{2} \ln t - 1)^2 }{4t} }  \\
& < & \frac{e^{1} e^{h(t)}}{t^{\varepsilon + 3/2}}  \left(1 +  \frac{\ln t}{t}  \right) - \frac{ e^{1}}{2 t^{5/2}},\\
& < & \frac{e^{1} e^{h(t)}}{t^{\varepsilon + 3/2}}  - \frac{ e^{1}}{2 t^{5/2}} + \frac{e^{1} e^{h(t)} \ln t}{t^{\varepsilon+ 5/2}} ,\\
& < & \frac{e^{1} e^{h(t)}}{t^{\varepsilon + 3/2}} - \frac{e^{1}}{4 t^{5/2}},
\end{eqnarray*}
where the inequalities hold for $t$ large enough. Notice that the term $-\frac{e^{1}}{2t^{5/2}}$ basically comes from the Dirichlet component in~\eqref{eq:defw_subsol0}. Moreover, it dominates the remainder~$ \frac{e^{1} e^{h(t)} \ln t}{t^{\varepsilon + 5/2}}$ thanks to~$\varepsilon >0$, which originates from our slowing down of the growth of the Neumann component in~\eqref{eq:defw_subsol0}. This is in turn the reason why we need to shift the traveling wave~$U_*$ slightly more than the expected logarithmic drift.

Now we finally take $ r \in \left( \frac{1}{2}, 1\right)$ and
$$\varepsilon = r - 1/2,$$
so that $ r+1 = \varepsilon + 3/2$. It follows that, up to increasing $T>0$ and for all $t \geq T$,
$$\underline{u}_1 (t, 2t + \ln t - 1 ) > \underline{u}_2 (t, 2t +  \ln t -1).$$
In a similar fashion, we have on the one hand
$$\underline{u}_1 (t, 2 t + \ln t + 1) = U_*( (r+1) \ln t + 1 - h(t) ) = \frac{e^{-1} e^{h(t)}}{t^{r+1 }} + O \left(\frac{1}{t^3}\right).$$
And on the other hand,
\begin{eqnarray*}
\underline{u}_2 (t, 2t + \ln t +1)& =&  \frac{1}{t^\varepsilon} \left( 1 + \frac{C}{\sqrt{t}} \right) \left( 1 + \frac{ \frac{3}{2}\ln t + 1}{t^{1-\varepsilon}} \right) \frac{e^{-1}}{t^{3/2}}    e^{- \frac{ ( \frac{3}{2} \ln t + 1)^2 }{ 4t}} \\
& = &    \frac{e^{-1} e^{h(t)} }{t^{\varepsilon + 3/2}}    e^{ -  \frac{3 \ln t}{4t} } e^{ - \frac{1}{ 4t}}  +   \left( 1 + \frac{C}{\sqrt{t}} \right) \frac{e^{-1}}{t^{5/2}}e^{- \frac{( \frac{3}{2}\ln t + 1)^2}{ 4t}}  \\
&  >& \frac{e^{-1} e^{h(t)}}{t^{\varepsilon + 3/2}} -  \frac{e^{-1} e^{h(t)} \ln t }{ t^{\varepsilon + 5/2}} + \frac{e^{-1}}{2 t^{5/2}},\\
& > & \frac{e^{-1} e^{h(t)}}{t^{\varepsilon + 3/2}} + \frac{ e^{-1}}{4 t^{5/2}},
\end{eqnarray*}
for $t$ large enough. Up to increasing $T$ again, it follows that for all $t \geq T$,
$$ \underline{u}_1 (t , 2t + \ln t + 1 ) < \underline{u}_2 (t, 2t + \ln t + 1).$$
Finally, we may define the following (generalized) subsolution:
$$\underline{u} (t,x) = \left\{\begin{array}{ll}
 \underline{u}_1 (t,x) & \mbox{ if } x \leq 2t + \ln t -1  , \vspace{3pt}\\
 \max \{ \underline{u}_1 (t,x)  ,\underline{u}_2 (t,x) \}  & \mbox{ if } 2t +  \ln t -1 < x < 2t + \ln t + 1, \vspace{3pt}\\
\underline{u}_2 (t,x) & \mbox{ if } x \geq 2t + \ln t  +1  .
\end{array}
\right.$$
More precisely, this is a subsolution for all $t \geq T$ and $x \in \R$, for some $T >0$ large enough. We highlight that a maximum (not a minimum) of two subsolutions creates another subsolution, which is why parameters had to be tuned carefully to get the necessary inequalities and hence continuity at $x= 2t + \ln t \pm 1$.

Notice that from its definition in Lemma~\ref{lem:ansatz2}, the function $\underline{u}_2$ converges uniformly to $0$ as $t \to +\infty$ in the subdomain $x \geq 2t + \ln -1$. It is then straightforward from the definition of $\underline{u}_1$ and~\eqref{ht_fast1} that~$\underline{u}$ satisfies statement~$(i)$ of Proposition~\ref{fast_subsol}. One may also check statement~$(ii)$ of Proposition~\ref{fast_subsol}, by noticing that
$$\lim_{x \to -\infty } \underline{u}_1 (T,x) =1 - e^{\frac{f'(1)}{2} T}<1,$$
and
$$\lim_{x \to + \infty } \frac{(x-X_2) \underline{u}_2 (T,x)}{e^{-(x-X_2)^2/4T}} = 0,$$
with $X_2 = 2T - \frac{1}{2} \ln T$. This concludes the proof of Proposition~\ref{fast_subsol}.

\subsection{A useful lemma}

In the KPP case, the nonlinearity $f$ is sublinear and therefore it is enough to consider the linearized equation to construct supersolutions. This is no longer the case for a general monostable reaction term where the nonlinear feature of \eqref{eq:main} cannot be discarded that easily. Therefore our supersolutions will be modeled on the traveling wave with minimal speed $U_*$ of the nonlinear problem. Still, because we must place ourselves in a moving frame with logarithmic drift, the function~$U_*$ cannot be used directly as a supersolution and one must instead approach it by solving an approximate nonlinear ODE. 

This is the purpose of the following lemma:
\begin{lem}\label{lem:approx}
Assume that the functions $U$ and $U_\varepsilon$ solve on the positive half-line, respectively,
$$U''+ 2 U' + f(U) = 0,$$
$$U_\varepsilon '' + (2 - \varepsilon) U_\varepsilon ' + f (U_\varepsilon) = 0,$$
as well as
$$U(0) - U_\varepsilon( 0) = O (\varepsilon), \qquad U '(0) - U_\varepsilon ' (0) = O (\varepsilon),$$
as $\varepsilon \to 0$. We also assume that
\begin{equation}\label{approx_eq1}
U, U' =  O (z e^{-z})
\end{equation}
as $z \to +\infty$.

Then, for any $\eta \in (0,1/2)$, there exists $\varepsilon_\eta$ such that for all $0 < \varepsilon < \varepsilon_\eta$ and $z \geq 0$,
\begin{equation}\label{wanted0}
 \max \{ | U(z) -U_\varepsilon (z) | ; |U ' (z) - U_\varepsilon ' (z) | \}  \leq \left( \varepsilon e^{-z} \right)^{1-\eta}.
\end{equation}
\end{lem}
\begin{rmk}\label{new_remark}
By our assumptions, $2$ is the minimal traveling wave speed. In particular, the function $U_\varepsilon$ does not qualify as a traveling wave, and this is due to the fact that it is not positive on the whole real line as required in~\eqref{TW_ineq}. As a matter of fact, by standard ODE technics one may show that for any $\varepsilon >0$, the function~$U_\varepsilon$ oscillates around~$0$ at $+\infty$.
\end{rmk}
\begin{proof}
By the standard stability theory of solutions of an ODE (see for instance~\cite{Gronwall}), we know that for any $Z>0$, there exists $C_Z >0$ such that
\begin{equation}\label{stability}
|U (z) - U_\varepsilon (z)| + |U '(z) - U_\varepsilon ' (z)| \leq C_Z \varepsilon,
\end{equation}
for all small $\varepsilon >0$ and $z \in [0,Z]$. In particular, for any~$Z>0$ and $\eta \in (0,1/2)$, one can choose $\varepsilon$ small enough so that the wanted inequality~\eqref{wanted0} holds on the interval~$[0,Z]$.

Next, we consider the function $V = e^{z} U$, which solves
$$V ''  - V + e^{z} f(Ve^{-z}) = 0.$$
Similarly $V_\varepsilon = e^z U_\varepsilon$ solves
$$V_\varepsilon '' -\varepsilon V_\varepsilon ' - (1-\varepsilon) V_\varepsilon + e^{z} f(V_\varepsilon e^{-z}) = 0.$$
Then we look at the difference $W := V - V_\varepsilon$, which satisfies
$$W '' - W + \varepsilon V_\varepsilon '  - \varepsilon V_\varepsilon + e^{z} f(V e^{-z}) - e^{z} f (V_\varepsilon e^{-z}) = 0.$$
On the other hand,
\begin{eqnarray*}
 f (V_\varepsilon e^{-z})  - f (V e^{-z}) &=& f' ( V e^{-z}) (V_\varepsilon - V) e^{-z} + \alpha_\varepsilon (V_\varepsilon - V)^2 e^{-2z}\\
& = &  \left( f'(0) +  \alpha_0 Ve^{-z} \right) (V_\varepsilon - V) e^{-z} + \alpha_\varepsilon (V_\varepsilon - V)^2 e^{-2z}\\
 & = &  (V_\varepsilon - V) e^{-z}  + \alpha_0  V (V_\varepsilon - V) e^{-2 z} + \alpha_\varepsilon (V_\varepsilon - V)^2 e^{-2z}\\
 & = &  - W e^{-z}  - \alpha_0  V W e^{-2 z} + \alpha_\varepsilon W^2 e^{-2z},
\end{eqnarray*}
where 
\begin{equation}\label{revised_alpha}
|\alpha_0|, | \alpha_\varepsilon | \leq \| f'' \|_\infty .
\end{equation}
Putting this in the equation for $W$ we get
$$W '' =  \varepsilon (V - V' ) + \varepsilon (W' - W)  - \alpha_0 V W e^{-z} + \alpha_\varepsilon W^2 e^{-z} .$$
Now take $\eta \in (0,1/2)$, and let us introduce 
$$\overline{W} (z) := \varepsilon^{1-\eta} e^{\eta z}.$$
We also define $Z_\eta \geq 1$ such that
\begin{equation}\label{eq:lemma_cond1}
\forall z \geq Z_\eta, \quad  C z  e^{-\eta z} \leq  \frac{\eta^2}{4},
\end{equation}
with
\begin{equation}\label{revised_C}
 C:= \left( 1+ \| f'' \|_\infty \right) \times \left( 1 +  \sup_{z \geq 1} \frac{|V (z)| + | V ' (z)|}{z}  \right) . 
\end{equation}
The fact that $C$ is a well-defined positive real number comes from the assumption~\eqref{approx_eq1}.

From~\eqref{stability} and $W (z) = e^{z} ( U (z) - U_\varepsilon (z))$, we can choose $\varepsilon_\eta$ small enough so that, for any $\varepsilon \leq \varepsilon_\eta$ and for all $z \in [0, Z_\eta]$,
\begin{equation}\label{stability_bis}
|W  (z) | + |W ' (z)| \leq   C_{Z_\eta} \varepsilon e^{z} < \eta  \varepsilon^{1-\eta} e^{\eta z }.
\end{equation}
In particular we have
$$| W(Z_\eta) | < \overline{W} (Z_\eta) \, , \quad  |W ' (Z_\eta)| < \overline{W}  ' (Z_\eta).$$
We will now apply a comparison argument to infer that 
$$- \overline{W} (z) < W( z) < \overline{W} (z) $$
for all $z \geq Z_\eta$. We argue by contradiction and assume that $W$ intersects either $- \overline{W}$ or~$\overline{W}$. We denote by $z_1$ the leftmost contact point on $(Z_\eta, +\infty)$. Without loss of generality we consider the case when $W (z_1)= \overline{W} (z_1)$, and due to $W ' (Z_\eta) < \overline{W}' (Z_\eta)$ we get that $\overline{W} - W$ reaches a positive maximum at some $z_0 \in (Z_\eta, z_1)$.
Then
$$\overline{W} '' (z_0 ) \leq W '' (z_0) .$$
Putting this together with the equation satisfied by $W$, we get
$$\overline{W} '' (z_0) \leq  \varepsilon (V - V' ) (z_0) + \varepsilon (W' - W) (z_0)  + \left[ \alpha_\varepsilon W (z_0)^2  - \alpha_0 V (z_0) W (z_0) \right] e^{-z_0} .$$
Moreover, we also have that
$$\overline{W} ' (z_0) = W ' (z_0) \  , \quad - \overline{W} (z_0) \leq W (z_0) \leq \overline{W} (z_0).$$
Thus, using also~\eqref{revised_alpha} and~\eqref{revised_C} we deduce that
\begin{eqnarray*}
\overline{W} '' (z_0) & \leq &  \varepsilon ( |V| + |V'|) (z_0) + \varepsilon ( \overline{W} ' + \overline{W} ) (z_0) + \left[ \alpha_\varepsilon \overline{W} (z_0)^2 + \alpha_0 |V(z_0)| \overline{W} (z_0) \right] e^{-z_0} \\
& \leq & C \varepsilon z_0 + \varepsilon ( \overline{W} ' + \overline{W}) (z_0) + C e^{-z_0} \left[ \overline{W} (z_0)^2 +  z_0 \overline{W} (z_0) \right].
\end{eqnarray*}
Since $\overline{W}(z) = \varepsilon^{1-\eta} e^{\eta z}$, we have that
\begin{eqnarray*}
\varepsilon^{1-\eta} \eta^2 &  \leq & C \varepsilon z_0 e^{- \eta z_0}+ \varepsilon^{2-\eta} (1+\eta) + C e^{-(1+\eta) z_0 } (  \varepsilon^{2- 2\eta} e^{2 \eta z_0} +\varepsilon^{1-\eta} z_0 e^{\eta z_0}  ) \\
&  \leq & C \varepsilon z_0 e^{- \eta z_0}+ \varepsilon^{2-\eta} (1+\eta) + 2 C z_0  e^{-(1-\eta) z_0 } \varepsilon^{1-\eta}.
\end{eqnarray*}
Recalling~\eqref{eq:lemma_cond1} and that $\eta \in (0,1/2)$, $z_0 > Z_\eta$, we get that
$$2 C z_0 e^{-(1-\eta) z_0)} \varepsilon^{1-\eta} \leq \varepsilon^{1-\eta} \frac{\eta^2}{2} ,$$ 
and then
$$ \varepsilon^{1-\eta} \frac{\eta^2 }{2} \leq \varepsilon \frac{\eta^2}{4} + \varepsilon^{2-\eta} (1+ \eta)  .$$
Up to reducing $\varepsilon$, this is a contradiction and we have proved that 
$$|W  (z) | \leq  \overline{W} (z) = \varepsilon^{1-\eta} e^{ \eta z}$$
for all $z \geq Z_\eta$. The same argument also shows that
$$ | W ' (z) | \leq \overline{W} ' (z) = \eta \varepsilon^{1-\eta} e^{\eta z}$$
for all $z \geq Z_\eta$. Indeed, we know that $W' (Z_\eta) < \overline{W} ' (Z_\eta)$ and, if there exists some $z_1 > Z_\eta $ such that $W ' (z_1) > \overline{W} ' (z_1)$, then $\overline{W}- W$ must reach a positive maximum at some $z_0 \in (Z_\eta, z_1)$. We have just shown that such a maximum cannot happen.

Putting this together with~\eqref{stability_bis} and $(U-U_\varepsilon) (z) = e^{-z} W(z)$, we find for $\varepsilon$ small enough that
$$|U(z) - U_\varepsilon (z)| \leq ( \varepsilon e^{-z})^{1-\eta}, \quad |U' (z) -U_\varepsilon ' (z)| \leq (1+\eta) (\varepsilon e^{-z})^{1-\eta},$$
for all $z \geq 0$. Since~$\eta$ can be arbitrarily chosen in $(0,1/2)$, we may replace it by~$\eta/2$ in the previous inequalities, and up to reducing~$\varepsilon$ again, one finally recovers that
$$\max \{ |U(z) - U_\varepsilon (z)| ; | U'(z) -U_\varepsilon '(z)| \} \leq \left(1 + \frac{\eta}{2} \right) ( \varepsilon e^{-z} )^{1-\frac{\eta}{2}}  \leq (\varepsilon e^{-z} )^{1-\eta},$$
for all $z \geq 0$. This concludes the proof of this lemma.
\end{proof}

\subsection{The fast decay case: a supersolution}\label{sec:fast_super}

We again assume that $U_*$ has fast decay in the sense that 
\begin{equation}\label{eq:asymp00}
U_* (z) = e^{-z} + O (e^{-(2-\eta) z}),
\end{equation}
as $z \to +\infty$, for any $\eta >0$; see also~\eqref{eq:asymp}. Our purpose is now to construct a supersolution whose level sets are located around $2t -r \ln t$, where
$$r <\frac{1}{2},$$
to be chosen close to $\frac{1}{2}$.
\begin{prop}\label{fast_supersol}
For any $r \in \left( \frac{1}{4} , \frac{1}{2} \right)$, there exist $T>0$ and $\overline{u} (t,x)$ a supersolution of~\eqref{eq:main} on $[T,+\infty) \times \R$ such that:
\begin{itemize}
\item[$(i)$] for any $\lambda \in (0,1)$, there exists $X_1 >0$ such that, for any $t$ large enough,
$$\forall x < 2 t - r \ln t - X_1, \quad \overline{u} (t,x) > \lambda,$$
$$\forall x > 2 t - r \ln t + X_1, \quad \overline{u} (t,x) < \lambda ;$$
\item[$(ii)$] the function $\overline{u}$ is positive on $[T,+\infty) \times \R$ and there exists $X_2> 0 $ such that
$$\forall x \leq -X_2, \quad \overline{u} (T,x) \geq 1.$$
\end{itemize}
\end{prop}
As before, the idea is to glue together different supersolutions. However, quite unusually, the construction of the supersolution turns out to be more complicated than the subsolution. One reason is that, if we slow down the traveling wave~$U^*$ to match the drift, it will not make a supersolution. Some perturbation is needed instead and this is where Lemma~\ref{lem:approx} is used.\\

We first introduce $\delta \in \left( 0, \frac{1}{2} \right)$ and $Z$ so that
\begin{equation}\label{delta00}
U_*(-Z) = 1 - 2\delta, \quad \mbox{ with} \ \ f ' (u) \leq \frac{f'(1)}{2} < 0 \mbox{ for } u \geq 1 - 2\delta.
\end{equation}
Up to decreasing $\delta$, we can also assume that $Z >0$ and 
\begin{equation}\label{eq:U''}
U_* '' (z) = - 2 U_* ' (z) - f (U_* (z)) < 0,
\end{equation}
for all $z \leq -Z$. This holds thanks to the fact that, as $z \to -\infty$, we have
$$U_* (z) = 1-  A e^{ \left(-1  + \sqrt{1 -  f' (1) }\right)z} + o \left(e^{\left(-1  + \sqrt{1 -  f' (1) }\right)z}  \right),$$
$$U_* ' (z) =  A \left(1 - \sqrt{1 - f' (1)}\right)  e^{\left(-1  + \sqrt{1 -  f' (1) }\right)z} + o \left(e^{\left(-1  + \sqrt{1 -  f' (1) }\right)z}  \right),$$
for some $A>0$, by standard ODE theory~\cite{CoddingtonLevinson}.

Then, for any $\gamma >0$, we define a function~$\overline{u}_1$ by,
\begin{equation}\label{def:overlineu1}
\forall t> \frac{1-\delta}{\gamma}, \, \forall x \leq 2t - r \ln t + h(t) -Z_\delta (t), \quad \overline{u}_1 (t,x) := U_*(x - 2t + r \ln t - h(t)) + \frac{\gamma}{t} ,
\end{equation}
where $h(t)$ is to be specified later but is already assumed to satisfy~\eqref{fast_subsol1}, i.e.
$$h(t) = O (1), \quad h '(t) = o(1/t) ,$$
as $t \to +\infty$, and $Z_\delta (t) $ is such that
\begin{equation}\label{eq:def_Z_delta}
U_* (-Z_\delta (t) ) = 1 - \delta - \frac{\gamma}{t} .
\end{equation}
In particular,
\begin{equation*}
\overline{u}_1 (t, 2t - r \ln t + h(t) - Z_\delta (t) ) = 1 - \delta .
\end{equation*}
The point $Z_\delta (t)$ is uniquely defined due to the fact that $U_* ' < 0$. The latter also implies that $Z_\delta (t)$ is an increasing function of $t$ and that
$$Z_\delta (t) > Z,$$
for any $t$ large enough, where we recall that $Z$ is such that $U_* (-Z) = 1-2\delta$ and~\eqref{eq:U''} holds for any $z \leq -Z$. Moreover,  it follows from the implicit function theorem that $Z_\delta$ is $C^1$ with respect to $t$ large enough, and that
\begin{equation}\label{claim00_add}
Z_\delta ' (t) = O \left( \frac{\gamma}{t^2} \right),
\end{equation}
as $t \to +\infty$. For later use, we also point out that $Z_\delta (t) \to Z_\delta^\infty$ as $t \to +\infty$, where 
$$U_* (-Z_\delta^\infty) = 1 -\delta .$$
As a consequence of~\eqref{claim00_add}, we even have that
\begin{equation}\label{claim00_add_useful}
Z_\delta (t) - Z_\delta^\infty = O \left( \frac{1}{t} \right),
\end{equation}
hence
\begin{equation}\label{add_again}
U_* ' (-Z_\delta (t)) - U_* ' (-Z_\delta^\infty) = O \left( \frac{1}{t}\right)
\end{equation}
as $t \to +\infty$.

The next lemma states that \eqref{def:overlineu1} defines a supersolution:
\begin{lem}\label{lem:fast_super_ansatz1}
If $h' (t) = o(1/t)$ as $t \to +\infty$, then there exist $\gamma >0$ and $T > \frac{1-\delta}{\gamma}$ such that the function~$\overline{u}_1$ defined by~\eqref{def:overlineu1} is a supersolution of \eqref{eq:main} for all $t \geq T$ and $x \leq 2t - r \ln t + h(t) -Z_\delta (t)$.
\end{lem}
\begin{proof}
For any $x \leq 2t - r \ln t + h(t) -Z_\delta (t)$, by~\eqref{eq:def_Z_delta} and the monotonicity of~$U_*$ we have that~$\overline{u}_1 (t,x) \geq 1 -\delta$. Provided~$T$ is large enough, for $t \geq T$ we also get that $U_* (x-2t + r\ln t -h(t)) \geq 1 - 2 \delta$. Then thanks to~\eqref{delta00} we compute
\begin{eqnarray*}
&& \partial_t \overline{u}_1  - \partial_{xx} \overline{u}_1 - f (\overline{u}_1) \\
& \geq & \partial_t \overline{u}_1  - \partial_{xx} \overline{u}_1 - f (U_* ) - \frac{f'(1)}{2} \frac{\gamma}{t} \\
& \geq & -\frac{\gamma}{t^2}   - \frac{f'(1)}{2} \frac{\gamma}{t} + \left( \frac{r}{t}  - h'(t) \right)  U_* '\\
& \geq &  - \frac{f'(1)}{4} \frac{\gamma}{t} + \left( \frac{r}{t}  - h'(t) \right)  U_* ' ,
\end{eqnarray*}
where the last inequality holds for $t$ large enough and the function $U_*$ is evaluated at $x - 2t + r \ln t  - h(t)$. Recalling that $f '(1) < 0$, $h'(t)=o(1/t)$ and that $U_* ' $ is uniformly bounded on the whole real line (it converges to 0 as $z \to \pm \infty$), we can find $\gamma$ large enough so that the above is nonnegative, i.e. $\overline{u}_1$ is a supersolution on the wanted subdomain.
\end{proof}
Then we extend $\overline{u}_1$ on the right of the point $2t - r \ln t + h(t) - Z_\delta (t)$, by letting
\begin{equation}\label{eq:fast_super_ansatz1bis}
\forall t \geq  T, \, \forall x > 2t - r \ln t +  h(t) -Z_\delta (t), \quad \overline{u}_1 (t, x) := U_r  (x - 2 t + r \ln t - h(t) + Z_\delta (t); t),
\end{equation}
where, for any $t \geq T$, the function $U_r (\cdot; t) $ is the solution of the ODE
\begin{equation}\label{newODE}
U_r '' + \left(2 - \frac{2 r}{t} \right) U_r ' + f (U_r) = 0, 
\end{equation}
on $[0,+\infty)$, together with the boundary conditions
\begin{equation}\label{newODE_boundary}
U_r (0 ;t) = 1 - \delta,  \quad U_r ' ( 0; t) = U_* ' (-Z_\delta (t) ). 
\end{equation}
Here the prime denotes the derivative with respect to the first variable, the other variable $t$ acting as a parameter of the ODE~\eqref{newODE}. To avoid any ambiguity, from hereafter we will denote by $z$ the first variable of $U_r$, hence by $\partial_z$ the corresponding derivative.

Notice that \eqref{newODE_boundary} ensures that~$\overline{u}_1$, as defined by~\eqref{def:overlineu1} and~\eqref{eq:fast_super_ansatz1bis}, and its spatial derivative $\partial_x \overline{u}_1$, are continuous functions. \\

The function $\overline{u}_1$ will serve as the first supersolution. However, notice that $2 - \frac{2r}{t} $ is below the minimal wave speed $ 2$. Therefore (see also Remark~\ref{new_remark}), due to the positivity of~$f$ in the interval $(0,1)$ and by a phase plane analysis, there exists $Z_0 (t)$ such that
\begin{equation}\label{eq:Z0_neg_deriv00}
 \partial_z  U_r (z; t) < 0,
\end{equation}
on the interval $[ 0,  Z_0 (t) ]$, with 
\begin{equation}\label{eq:Z0_neg00}
U_r (Z_0 (t); t) = 0 .
\end{equation}
Thus, one may already expect that a second supersolution will be necessary. Before we construct it, let us check that the newly defined function~$\overline{u}_1$ is again a supersolution.
\begin{lem}\label{lem:fast_super_ansatz1bis}
If $h'(t) = o(1/t)$ as $t \to +\infty$, then there exist $\gamma >0$ and $T >0$ such that the function~$\overline{u}_1$ defined by~\eqref{def:overlineu1} and~\eqref{eq:fast_super_ansatz1bis} is a supersolution of
\eqref{eq:main} for all $t \geq T$ and $x < 2t - r \ln t +h (t) - Z_\delta (t) + Z_0 (t)$.
\end{lem}
\begin{proof}
We have already proved that $\overline{u}_1$ is a supersolution for $t \geq T$ and $x \leq 2 t - r \ln t + h(t) -Z_\delta (t)$, which was Lemma~\ref{lem:fast_super_ansatz1}. Then, since both $\overline{u}_1$ and $\partial_x \overline{u}_1$ are continuous at $x = 2 t - r \ln t + h(t) -Z_\delta (t)$, it is enough to check that $(t,x) \mapsto U_r (x-2t + r \ln t - h(t) + Z_\delta (t); t)$ satisfies the wanted differential inequality for $0 < x - 2t + r \ln t - h(t) + Z_\delta (t) < Z_0 (t)$.

Let us briefly consider $t$ as a parameter in the ODE~\eqref{newODE}satisfied by $U_r (\cdot; t)$. By the standard regularity theory of ODEs~\cite{Gronwall}, the function $t \mapsto U_r (\cdot; t)$ admits a derivative. Furthermore, we claim that, for any $T \leq t_1 < t_2$, 
\begin{equation}\label{claim00}
\forall z \in (0, Z_0 (t_1) ), \quad U_r (z ; t_1) < U_r (z ;  t_2).
\end{equation}
To check this, first recall that $Z_\delta (t)$ is an increasing function of $t$ and that it is bounded from below by $Z$ for $t \geq T$, up to increasing~$T$; see~\eqref{eq:def_Z_delta} and the subsequent discussion. Recalling also \eqref{eq:U''} and that $\partial_z U_r  (0;t) = U_* ' (-Z_\delta (t))$, we get that
$$t \mapsto \partial_z U_r  (0;t)$$
is a negative and increasing function. Since $U_r (0;t_1)  = U_r (0;t_2) = 1-\delta$, it follows that the wanted inequality \eqref{claim00} holds on a right neighborhood of $z =0$.

Now proceed by contradiction and assume that \eqref{claim00} does not hold. In that case there must exist $z_1 \in (0, Z_0 (t_1))$ such that
$$U_r (\cdot; t_1) < U_r (\cdot; t_2)$$
in the interval $(0,z_1)$, and $0 < U_r (z_1 ;t_1) = U_r (z_1 ;t_2)$. In particular~$U_r (\cdot ; t_2) >0$ on $(0,z_1]$ hence $z_1 \in (0, Z_0 (t_2))$.

From the monotonicity property~\eqref{eq:Z0_neg_deriv00}, we get
$$U_r (0;t_1) > U_r (z_1 ;t_2).$$
Then we define
$$S^* := \inf \{ S \geq 0 \, | \ U_r (\cdot - S; t_1) \geq U_r (\cdot; t_2) \ \mbox{ in $(S, z_1)$ } \} \in (0, z_1).$$
By construction, we have that $U_r (\cdot -S^* ;t_1) - U_r (\cdot;t_2) \geq 0 $ in $[ S^*,z_1 ]$, and that there exists $z_2 \in [ S^* ,z_1 ]$ such that 
$$U_r (z_2-S^* ;t_1) - U_r (z_2;t_2) =0.$$
Furthermore, using again the monotonicity property~\eqref{eq:Z0_neg_deriv00} for both $U_r (\cdot; t_1)$ and $U_r (\cdot; t_2)$, we have $U_r (0; t_1)  = U_r (0;t_2 )> U_r (S^*; t_2)$ and $U_r (z_1 -S^* ;t_1 ) > U_r (z_1; t_1) = U_r (z_1;t_2)$. Thus $z_2 \in (S^* ,z_1)$ and
$$\partial_{z} U_r (z_2 - S^* ; t_1) = \partial_{z} U_r (z_2 ; t_2) , \quad \partial_{zz} U_r (z_2 - S^* ;t_1) \geq \partial_{zz} U_r (z_2;t_2).$$
However, from the ODEs~\eqref{newODE} solved by $U_r (\cdot; t_1)$ and $U_r (\cdot;t_2)$, and the fact that $t_1 < t_2$, one may now check that
\begin{eqnarray*}
\partial_{zz} U_r  (z_2 -S^*;t_1) -   \partial_{zz} U_r  (z_2;t_2) & = & \frac{2r}{t_1} \partial_z   U_r (z_2 - S^*;t_1) - \frac{2r}{t_2} \partial_z U_r (z_2,t_2) \\
& = & \left( \frac{2r}{t_1}  - \frac{2r}{t_2} \right) \partial_z U_r (z_2,t_2) \\
& < & 0.
\end{eqnarray*}
We have reached a contradiction and this proves Claim~\eqref{claim00}.

It now follows from Claim~\eqref{claim00} that
$$\partial_t U_r (z;t) \geq 0,$$
for all $t \geq T$ and $z \in (0,Z_0(t))$.
We can now compute, for $t >0$ and $x -2t +r \ln t - h(t) + Z_\delta (t) \in (0,Z_0(t))$,
\begin{eqnarray*}
&& \partial_t \overline{u}_1  - \partial_{xx} \overline{u}_1 - f (\overline{u}_1) \\
& = & \partial_t U_r  - \partial_{zz} U_r  - \left( 2 - \frac{r}{t}  + h' (t) - Z_\delta ' (t) \right) \partial_z U_r  - f (U_r) \\
& \geq &  \left( - \frac{r}{t} - h' (t)  + Z_\delta ' (t) \right) \partial_z U_r  \\
& \geq & 0,
\end{eqnarray*}
where $U_r$ and its derivatives are evaluated at $(x - 2t + r \ln t - h(t) + Z_\delta (t) ,t)$ and the last inequality holds for large times, thanks to \eqref{claim00_add}, \eqref{eq:Z0_neg_deriv00} and our assumption that~$h' (t) = o(1/t)$ as $t \to +\infty$.
\end{proof}

As we already mentioned, the function $U_r (z;t)$ (hence $\overline{u}_1$) changes sign, which is why a second supersolution will be needed. Before we proceed, we need a better understanding of $U_r (z;t)$ as $z \to +\infty$, or more precisely around $(r+1) \ln t$ where we will match it with our second supersolution. First recall that~$U_r (\cdot ; t)$ solves~\eqref{newODE} and~\eqref{newODE_boundary}. Therefore, by~\eqref{add_again} and Lemma~\ref{lem:approx} with $\varepsilon = \frac{2r}{t}$, $U = U_* (\cdot - Z_\delta^\infty)$ and $U_\varepsilon = U_r ( \cdot ;t)$, we get for any small $\eta >0$, any $z \geq 0$ and any $t$ large enough that
\begin{equation}\label{revised_decay0}
|U_r (z;t) - U_* (z- Z_\delta^\infty )| \leq \left( \frac{2r}{t} e^{-z} \right)^{1-\eta}.
\end{equation}
In particular this gives
\begin{equation}\label{eq:ansatz1_decay1}
\sup_{z \geq (r+1) \ln t + Z_\delta^\infty - 2} |U_r  ( z ;t) - U_*(z - Z_\delta^\infty ) | = O \left(  t^{- (r + 2)(1-\eta) }\right),
\end{equation}
as $t \to +\infty$, with $\eta>0$ arbitrarily small. Moreover, recall from~\eqref{eq:Z0_neg_deriv00}-\eqref{eq:Z0_neg00} that $Z_0 (t)$ is the smallest positive point where $U_r (\cdot ;t)$ equals~0. Then using again~\eqref{revised_decay0}, we get first that $Z_0 (t) \to +\infty$ as $t \to +\infty$, and then also that
$$U_* (Z_0 (t) - Z_\delta^\infty) \leq  \left( \frac{2r}{t} e^{-Z_0 (t)} \right)^{1-\eta}.$$
Due also to~\eqref{eq:asymp00}, we have
$$
\left( \frac{2r}{t} e^{-Z_0 (t)} \right)^{1-\eta} \geq e^{-(Z_0 (t) - Z_\delta^\infty)} - K e^{-(2-\eta) (Z_0 (t) - Z_\delta^\infty)},
$$
for some $K>0$ and $\eta >0$ arbitrarily small, from which one can infer that
\begin{equation}\label{eq:ansatz1_decay2}
Z_0 (t) > (r+2) \ln t  ,
\end{equation}
for $t$ large enough.\\

Let us now construct the second supersolution. As in the construction of the subsolution, we use some combination of Neumann and Dirichlet eigenfunctions of a linear operator obtained by some change of variables in an appropriate moving frame. 

Let us recall the equation
\begin{equation}\label{eq:rewritten}
\partial_\tau w = \partial_{yy} w - \left( \frac{y^2}{16} - \frac{3}{4} \right) w - \frac{1}{2} e^{-\frac{\tau}{2}} \left( \partial_y w - \frac{y}{4} w \right),
\end{equation}
which is equivalent to \eqref{eq:main} in the moving frame centered at $2t- \frac{1}{2} \ln t$; see the computation leading to~\eqref{eq:nonautonomous} in Subsection~\ref{sec:fast_sub}. Notice that, plugging the first Neumann eigenfunction $e^{-\frac{y^2}{ 8}}$ in the right-hand term, then the nonautonomous term $\partial_y w - \frac{y}{4} w $ is negative which is the wrong sign when looking for a supersolution of~\eqref{eq:rewritten}. Therefore, here we need to proceed more carefully.

We first take the following combination:
\begin{equation}\label{eq:defw_supersol0}
w_0 (\tau, y ) := e^{-\frac{y^2}{8}} e^{\tau \left( \frac{1}{2} + \varepsilon \right)} -  2 y e^{-\frac{y^2}{8}}  e^{ 2 \varepsilon \tau} +  y^2 e^{-\frac{y^2}{8}} e^{2 \varepsilon \tau}.
\end{equation}
Here we fix
$$\varepsilon = \frac{1}{2} -r \in \left(0 , \frac{1}{2} \right),$$
to be made arbitrarily small.

The above combination~\eqref{eq:defw_supersol0} shares some similarities with the one used in our subsolution; see~\eqref{eq:defw_subsol0}. In particular, the first two functions $e^{-\frac{y^2}{8}}$, $y e^{-\frac{y^2}{8}}$ are the first eigenfunctions of the elliptic operator $\mathcal{L} w := \partial_{yy}w - \left(\frac{y^2}{16} - \frac{3}{4} \right) w$, posed on the positive half-line respectively with a Neumann and a Dirichlet boundary condition at~$0$, and the corresponding eigenvalues are respectively $1/2$ and $0$. However, in order to get a supersolution, we instead slightly increased the exponential growth in time of the Neumann component. Furthermore, since it is crucial to make the second supersolution steeper than~$\overline{u}_1$, we put a negative constant in front of the Dirichlet component. We also made the exponential growth in time of the Dirichlet component ($2\varepsilon$ instead of~$0$) even faster relatively to the Neumann part ($\varepsilon + 1/2$ instead of $1/2$), which will eventually ensure that the Dirichlet component is large enough to dictate the slope around the matching point of both supersolutions. Finally, the third term, which is modeled on the second Neumann eigenfunction of~$\mathcal{L}$, ensures the positivity of~$w_0$. More precisely, there exists $\tau_1 \geq 0$ such that, for any $\tau \geq \tau_1$ and $y \geq 0$, 
$$e^{ \tau \left( \frac{1}{2} - \varepsilon \right)}  -  2 y + y^2 > 0 ,$$
hence $w_0$ is positive on $[\tau_1 , +\infty) \times \R_+$.

Let us now check that $w_0$ is a supersolution of~\eqref{eq:rewritten} for $\tau \geq \tau_1$ (up to increasing~$\tau_1$) and $y \in [0,3]$. For convenience and since the equation is linear, we write $w_0$ as $w_1 + w_2 + w_3$, each $w_i$ denoting the $i$-th term of the sum in the right-hand side of~\eqref{eq:defw_supersol0}. We now compute the equation for each component separately.

First,
\begin{eqnarray*}
&& \partial_\tau w_1 - \partial_{yy} w_1 + \left( \frac{y^2}{16} - \frac{3}{4} \right) w_1 + \frac{1}{2} e^{-\frac{\tau}{2}} \left( \partial_y w_1 - \frac{y}{4} w_1 \right) \\
& = & \varepsilon w_1 + \frac{1}{2} e^{- \frac{\tau}{2}} \left( \partial_y w_1  - \frac{y}{4} w_1 \right) \\
& = & \varepsilon e^{-\frac{y^2}{8}} e^{\tau \left( \frac{1}{2} + \varepsilon \right) }  - \frac{y}{4} e^{-\frac{y^2}{8}} e^{\varepsilon \tau}\\
& \geq & \varepsilon e^{-\frac{y^2}{8}} e^{\tau \left( \frac{1}{2} + \varepsilon \right) }   - (1+ y^3) e^{-\frac{y^2}{8}} e^{2\varepsilon \tau}.
\end{eqnarray*}
Next,
\begin{eqnarray*}
&& \partial_\tau w_2 - \partial_{yy} w_2 + \left( \frac{y^2}{16} - \frac{3}{4} \right) w_2 + \frac{1}{2} e^{-\frac{\tau}{2}} \left( \partial_y w_2 - \frac{y}{4} w_2 \right) \\
& = & 2\varepsilon w_2 +  \frac{1}{2} e^{-\frac{\tau}{2}} \left( \partial_y w_2  - \frac{y}{4} w_2 \right)  \\
& =& - 4 \varepsilon y e^{-\frac{y^2}{8}} e^{2 \varepsilon \tau} + e^{-\frac{y^2}{8}} e^{-\frac{\tau}{2}} e^{2 \varepsilon \tau} \left( \frac{y^{2}}{2} -1 \right ) \\
& \geq &  - (4 \varepsilon y +1 ) e^{-\frac{y^2}{8}} e^{2\varepsilon \tau} \\ 
& \geq & -2 (1+y^3) e^{-\frac{y^2}{8}} e^{2 \varepsilon \tau}.
\end{eqnarray*}
Finally,
\begin{eqnarray*}
&& \partial_\tau w_3 - \partial_{yy} w_3 + \left( \frac{y^2}{16} - \frac{3}{4} \right) w_3 + \frac{1}{2} e^{-\frac{\tau}{2}} \left( \partial_y w_3 - \frac{y}{4} w_3 \right) \\
& = &  2 \varepsilon w_3 + \frac{1}{2} w_3  - 2 e^{-\frac{y^2}{8}} e^{2 \varepsilon \tau} +  \frac{1}{2} e^{-\frac{\tau}{2}} \left( \partial_y w_3  - \frac{y}{4} w_3 \right) \\
& = &   \left( 2 \varepsilon +  \frac{1}{2} \right) y^2 e^{-\frac{y^2}{8}} e^{2 \varepsilon \tau} - 2 e^{-\frac{y^2}{8}} e^{2 \varepsilon \tau} +  e^{-\frac{y^2}{8}} e^{-\frac{\tau}{2}} e^{2\varepsilon \tau} \left(  y - \frac{y^3}{4}   \right)\\
& \geq &  - \left(2 + \frac{y^3}{4} \right) e^{-\frac{y^2}{8}} e^{2 \varepsilon \tau} .
\end{eqnarray*}
Putting all the above together, we find that
\begin{eqnarray*}
& &\partial_\tau w_0 -\partial_{yy} w_0 + \left( \frac{y^2}{16} - \frac{3}{4} \right) w_0 + \frac{1}{2} e^{-\frac{\tau}{2}} \left( \partial_y w_0 - \frac{y}{4} w_0 \right)  \\
& \geq & \varepsilon e^{-\frac{y^2}{8}} e^{\tau \left( \frac{1}{2} + \varepsilon \right) } -  K_1   (1 + y^3) e^{-\frac{y^2}{8}} e^{2 \varepsilon \tau} \\
& > &  0,
\end{eqnarray*}
where $K_1$ is a positive constant which does not depend on $y \geq 0$ and $\tau \geq 0$, and the last inequality holds on the interval $y \in [0,3]$ for all $\tau \geq \tau_1$ (up to increasing~$\tau_1$). Unfortunately, looking at the negative sign of the single $y^3$-order term in the previous computations, one may observe that $w_0$ cannot be a supersolution on the whole right half-line~$\{y \geq 0\}$. Therefore, we need to merge it with yet another function. One may also notice that this $y^3$-order term comes from the inclusion of~$w_3$ in the definition of~$w_0$, and try to remove it to solve this issue. However, if we remove~$w_3$ then the function~$w_0$ is no longer positive so that an additional step is needed either way. Moreover, by adding this third term, we also make the decay of $w_0$ slower as $y \to +\infty$, which actually makes this additional step easier.

This leads us to also introduce $\tilde{w}$ which solves \eqref{eq:rewritten}, i.e.
$$\partial_\tau \tilde{w} -\partial_{yy} \tilde{w} + \left( \frac{y^2}{16} - \frac{3}{4} \right) \tilde{w} + \frac{1}{2} e^{-\frac{\tau}{2}} \left( \partial_y \tilde{w} - \frac{y}{4} \tilde{w} \right) = 0,$$
for $\tau >0$ and $y >0$, together with the Neumann boundary condition
$$ \partial_y \tilde{w} (\tau, 0) = 0,$$
for $\tau >0$, and the initial data
$$\tilde{w} (0, y) =  e^{-\frac{y^2}{8}},$$
for $y >0$. Proceeding similarly as in the proof of~\cite[Lemma 2.2]{HNRR}, one may check that there exists $W_1>0$, and for any bounded interval $[0,L]$ there exists $K (L) >0$, such that
\begin{equation}\label{eq:appen}
 \left| \tilde{w}  (\tau ,y) -  W_1   e^{-\frac{y^2}{8}}e^{\frac{\tau}{2}} \right| \leq  K(L)  ,
\end{equation}
for any $\tau \geq 1$ and $y \in [0,L]$. For the sake of completeness, we include the details in Appendix~\ref{appen}. By the strong maximum principle, the function~$\tilde{w}$ is positive. We also notice, for later use, that $ e^{\tau } e^{\frac{y^2 }{8}}$ is a supersolution of~\eqref{eq:rewritten}. and thus, for any $\tau \geq 0$ and $y \geq 0$,
\begin{equation}\label{fast_supersol_ansatz2_thing2}
0 < \tilde{w} (\tau  , y) \leq  e^{\tau } e^{\frac{y^2 }{8}}.
\end{equation}

Next we define, for $\tau \geq 0$,
\begin{equation}\label{def:revised_supersol_truc}
\overline{w} (\tau, y)  := \left\{
\begin{array}{ll}
\displaystyle w_0 (\tau,y) & \mbox{ if } \ 0 \leq y \leq 1, \vspace{3pt}\\
\displaystyle \min \left\{ w_0 (\tau, y) , \frac{1}{W_1} \tilde{w} (\tau , y) \times e^{\varepsilon \tau } \right\} & \mbox{ if }\  1 <  y < 3,\vspace{3pt}\\
\displaystyle  \frac{1}{W_1}  \tilde{w} (\tau , y) \times e^{\varepsilon \tau} & \mbox{ if } \ y \geq 3.
\end{array}
\right.
\end{equation}
By definition, this function is positive but coincides with $w_0$ when $y \leq 1$, which in the original variables writes as $x - 2t + \frac{1}{2} \ln t \leq \sqrt{t}$. In particular we will be able to compute the resulting supersolution explicitely at the matching point with $\underline{u}_1$, in a similar manner as we did in the proof of Proposition~\ref{fast_subsol}.

To check that it is a supersolution of~\eqref{eq:rewritten}, we must look at which of the two functions realize the minimum at $y=1$ and $y=3$. First, by~\eqref{eq:defw_supersol0} we have for $\tau$ large enough that
$$w_0 ( \tau,  1 ) - e^{-\frac{1}{ 8}} e^{ \tau \left( \frac{1}{2}  + \varepsilon \right)} = - e^{-\frac{1}{ 8}} e^{2 \varepsilon \tau} < - \frac{ K (3)}{W_1} e^{\varepsilon \tau} .$$
Using~\eqref{eq:appen}, we get 
$$w_0 ( \tau, 1) <  \frac{1}{W_1} \tilde{w} (\tau, 1) e^{\varepsilon \tau },$$
hence
$$\overline{w} (\tau,1)  = w_0 (\tau, 1),$$
for any $\tau$ large enough.

On the other hand, at $y=3$, we have by~\eqref{eq:appen} that
$$\frac{1}{W_1} \tilde{w} (\tau , 3) e^{\varepsilon \tau } \leq e^{-\frac{9}{8}} e^{\tau \left( \frac{1}{2} +\varepsilon \right)} + \frac{ K(3)}{W_1} e^{\varepsilon \tau},$$
while
$$w_0 (\tau,3 ) =e^{-\frac{9}{8}} e^{ \tau\left( \frac{1}{2} +\varepsilon \right)}  + 3 e^{-\frac{9}{8}}  e^{2 \varepsilon \tau}  .$$
Thus, we find that 
$$\overline{w} (\tau, 3) = \frac{1}{W_1} \tilde{w} (\tau ,3) e^{\varepsilon \tau},$$ for all~$\tau$ large enough.
Notice that the third term in the definition of~$w_0$ is what ensured that~$w_0 (\tau,\cdot)$ eventually always intersects $\frac{1}{W_1} \tilde{w} (\tau,y) \times e^{\varepsilon \tau}$ in the interval~$[1,3]$.

We infer that the function $\overline{w} (\tau,y)$ is continuous for $\tau$ large enough. Moreover, $w_0 (\tau,y)$ is a supersolution of~\eqref{eq:rewritten} on the domain $[\tau_1, +\infty) \times [0,3]$, and $\frac{1}{W_1} e^{\varepsilon \tau } \tilde{w} (\tau , y)$ is a supersolution on the whole domain $[0, +\infty) \times [0, +\infty)$ (by a straightforward computation). We finally conclude, up to increasing~$\tau_1$, that $\overline{w} (\tau,y)$ is a supersolution of~\eqref{eq:rewritten} on $[\tau_1 , +\infty) \times [0,+\infty)$.\\

%

Going back to the original problem, we are in position to prove the following lemma, which is the last one before gluing together our supersolutions.
\begin{lem}\label{lem:fast_super_ansatz22}
For any $\varepsilon \in \left( 0 , \frac{1}{4} \right)$, there exist $C>0$ and $T>0$ such that the function
$$\overline{u}_2 (t,x) := \left( 1 - \frac{C}{t^{1/4}} \right) \tilde{u} \left(t,x -2t + \frac{1}{2} \ln t \right),$$
is a supersolution of \eqref{eq:main} in the subdomain $t \geq T$ and $x \geq 2t + \ln t - 2$, where~$\tilde{u}$ is a positive function which satisfies that
\begin{equation}\label{eq:explicit_rev}
\tilde{u} (t,z) = t^\varepsilon \left[ 1 - 2 \frac{z}{t^{1-\varepsilon}} +  \frac{z^2}{t^{3/2 - \varepsilon }}  \right] e^{-z} e^{-\frac{z^2}{4t} },
\end{equation}
for any $t \geq T$ and $z \in [0, \sqrt{t}]$, as well as
\begin{equation}\label{eq:lem:fast_super_ansatz21_upper}
\tilde{u} (t,z) \leq t^{2 \varepsilon + \frac{1}{2}} e^{-z},
\end{equation}
for any $t \geq T$ and $z \geq 0$.
\end{lem}
%
\begin{proof}
We define
$$\tilde{u} (t,z) = e^{-z}  e^{-\frac{y^2}{8}} e^{-\frac{\tau}{2}} \times \overline{w} (\tau,y), $$
where $\tau = \ln t$, $y = \frac{z}{\sqrt{t}}$ and the (positive) function~$\overline{w}$ was introduced in~\eqref{def:revised_supersol_truc}. In particular, recalling also~\eqref{eq:defw_supersol0}, we immediately get an explicit formula for $\tilde{u} (t,z)$ when $0 \leq z \leq \sqrt{t}$ which is precisely~\eqref{eq:explicit_rev}. 

Moreover, it follows from~\eqref{def:revised_supersol_truc} together with~\eqref{fast_supersol_ansatz2_thing2} that there exists some $C>0$ such that
$$\overline{w} (\tau,y) \leq C  e^{(1+ \varepsilon) \tau} e^{\frac{y^2}{8}},$$
for all $\tau \geq 0$ and  $y > 1$. Up to increasing the constant~$C$, it is straightforward from~\eqref{eq:defw_supersol0} that the same inequality is also satisfied for $y\in [0,1]$. In terms of $\tilde{u}$, we infer that
$$\tilde{u} (t,z) \leq C t^{ \varepsilon + 1/2} e^{-z},$$
for any $t \geq 1$ and $z \geq 0$. Choosing any~$T$ large enough, we get~\eqref{eq:lem:fast_super_ansatz21_upper} for all~$t\geq T$.

Next, it follows from the discussion preceding Lemma~\ref{lem:fast_super_ansatz22} that~$\overline{w}$ is a supersolution of~\eqref{eq:rewritten} for $\tau \geq 0$ and $y \geq 0$, which is equivalent to 
$$\partial_t \tilde{u} - \partial_{zz} \tilde{u}  - \left( 2 - \frac{1}{2t } \right) \partial_z \tilde{u}  - \tilde{u}  \geq 0,$$
for any $t \geq 1$ and $z \geq 0$. Letting $K = \| f'' \|_\infty$, we compute
\begin{eqnarray*}
&& \partial_t \overline{u}_2 - \partial_{xx} \overline{u}_2 -  f(\overline{u}_2) \\
& \geq &  \partial_t \overline{u}_2 - \partial_{xx} \overline{u}_2 -  \overline{u}_2 -  K \overline{u}_2^2 \\
& \geq & \frac{C}{4 t^{5/4}} \tilde{u}   - K \left( 1 - \frac{C}{t^{1/4}} \right)^2 \tilde{u}^2 \\
& \geq & \tilde{u} \left( \frac{C}{4 t^{5/4}} - K \tilde{u}\right),
\end{eqnarray*}
where $\tilde{u}$ is evaluated at $(t,x -2t + \frac{1}{2 } \ln t)$.

Now, for all $x - 2t + \frac{1}{2} \ln t \in \left[ \frac{3}{2} \ln t - 2 ,  \sqrt{t} \right]$, we can use the explicit formula~\eqref{eq:explicit_rev} for~$\tilde{u}$. In particular, we find that 
\begin{eqnarray*}
\sup_{ \left[ \frac{3}{2} \ln t - 2 ,  \sqrt{t} \right]} \tilde{u} ( t,\cdot)  & \leq &  t^\varepsilon \left[1 + \frac{t}{t^{3/2 -\varepsilon}} \right] e^{2-\frac{3}{2} \ln t }\\
& = & O (t^{\varepsilon - \frac{3}{2}})
\end{eqnarray*}
as $t \to +\infty$. Therefore, thanks to $0 < \varepsilon  < \frac{1}{4}$ and up to increasing $T$, we get that
$$\partial_t \overline{u}_2 - \partial_{xx} \overline{u}_2 - f (\overline{u}_2) \geq 0,$$
for all $t \geq T$ and $2t  + \ln t - 2 \leq x \leq 2t +  \sqrt{t} - \frac{1}{2} \ln t $.

Then, for $  x \geq 2t +  \sqrt{t} - \frac{1}{2} \ln t$, from~\eqref{eq:lem:fast_super_ansatz21_upper} we get that 
$$\tilde{u} \left( t,x -2t + \frac{1}{2} \ln t \right) \leq t^{2 \varepsilon + \frac{1}{2}} e^{- \sqrt{t}}, $$
and the wanted differential inequality is again satisfied for large enough times. Finally we have proved Lemma~\ref{lem:fast_super_ansatz22}.
\end{proof}
Let us now proceed with the proof of Proposition~\ref{fast_supersol} and the ``matching'' argument of the two supersolutions $\overline{u}_1$ (recall Lemmas~\ref{lem:fast_super_ansatz1} and~\ref{lem:fast_super_ansatz1bis}) and $\overline{u}_2$ (see Lemma~\ref{lem:fast_super_ansatz22}). This matching occurs in a bounded neighborhood of $2t + \ln t$ (with respect to the original non moving frame), and it again shares some similarities with the construction of the subsolution. We choose
$$h(t) = \ln \left( 1 - \frac{C}{t^{1/4}} \right) + \ln \left( 1 - \frac{3 \ln t}{ t^{1-\varepsilon}} + \frac{9 (\ln t)^2}{4 t^{3/2 - \varepsilon}}  \right) - \frac{9 (\ln t)^2}{16 t} ,
$$
so that
$$e^{h(t)} =   \left( 1 - \frac{C}{t^{1/4}} \right) \left( 1 - \frac{3 \ln t}{ t^{1-\varepsilon}} + \frac{9 (\ln t)^2}{4 t^{3/2 - \varepsilon}}  \right) e^{- \frac{9 (\ln t)^2}{16 t} }.
$$
As before, we find that $h (t) \to 0$ as well as $h '(t) = o(1/t)$ as $t \to +\infty$. Thus, Lemmas~\ref{lem:fast_super_ansatz1} and~\ref{lem:fast_super_ansatz1bis} apply.

First we point out that the asymptotics~\eqref{eq:asymp00} of the traveling wave can be extended to its first derivative, i.e. we also have
$$U_* ' (z) = - e^{-z} + O (e^{-(2-\eta) z}),$$
as $z \to +\infty$. Thanks to these asymptotics, using also~\eqref{claim00_add_useful} and~\eqref{eq:ansatz1_decay1} and taking $\varepsilon = \frac{1}{2} - r\in \left( 0, \frac{1}{4} \right)$, we find that
\begin{eqnarray*}
\overline{u}_1 (2t + \ln t - 1) & = &  U_r ( (r+1) \ln t -1 - h(t) + Z_\delta (t); t) \\
 & = &  U_* ( (r+1)\ln t -1 - h(t) + Z_\delta (t) - Z_\delta^\infty ) + O \left(\frac{1}{t^{ r+2 - \frac{1}{2} \varepsilon}}\right) \\
 & = &  U_* ( (r+1)\ln t -1 - h(t)   ) + O \left(\frac{1}{t^{ r+2 - \frac{1}{2} \varepsilon}}\right) \\
& =&  \frac{e^1 e^{h(t)} }{t^{r+1}}  +  O \left(\frac{1}{t^{ r+2 - \frac{1}{2} \varepsilon}}\right) \\
& =&  \frac{e^1 e^{h(t)} }{t^{3/2 - \varepsilon}}  +  O \left(\frac{1}{t^{ \frac{5}{2} - \frac{3}{2} \varepsilon}}\right)
\end{eqnarray*}
as $t \to +\infty$. 
On the other hand,
\begin{eqnarray*}
\overline{u}_2 (t, 2t + \ln t -1)& =&  t^\varepsilon  \left( 1 - \frac{C}{t^{1/4}} \right) \frac{e^{1}}{t^{3/2}}    e^{- \frac{( \frac{3}{2} \ln t - 1)^2 }{ 4t}}   \\
& & \qquad  \times \left[ 1  - 2 \frac{ \frac{3}{2} \ln t -1 }{t^{1-\varepsilon}} +  \frac{\left( \frac{3}{2} \ln t -1 \right)^2}{t^{3/2 - \varepsilon} }\right]  \\
& = &    \frac{e^{1} e^{h(t)} }{t^{3/2 - \varepsilon }}    e^{ \frac{3 \ln t}{4t} -\frac{ 1 }{ 4t}}   \\
&& \qquad + \frac{e^1}{t^{3/2 - \varepsilon}}\left( 1 - \frac{C}{t^{1/4}} \right) \left[ \frac{2}{t^{1-\varepsilon}} +    \frac{-3 \ln t +1} {t^{3/2 - \varepsilon}} \right] e^{- \frac{\left( \frac{3}{2} \ln t -1 \right)^2 }{ 4t}} \\
& > & \frac{e^{1} e^{h(t)}}{t^{3/2 - \varepsilon}}  + \frac{ e^{1}}{t^{5/2 - 2 \varepsilon}} ,
\end{eqnarray*}
where as usual the inequality holds at large times. Notice that we used the explicit formula~\eqref{eq:explicit_rev} from Lemma~\ref{lem:fast_super_ansatz22} for $\overline{u}_2$, since $0 \leq \frac{3}{2} \ln t - 1 \leq  \sqrt{t}$ for large enough $t$.

It follows that
$$\overline{u}_2 (t, 2t + \ln t -1) > \overline{u}_1 (t , 2t + \ln t -1).$$
By similar computations, we get
$$\overline{u}_1  ( 2t  + \ln t + 1 ) = \frac{e^{-1} e^{h(t)}}{t^{3/2 - \varepsilon}} + O \left(\frac{1}{t^{\frac{5}{2} - \frac{3}{2} \varepsilon}}\right),$$
and
\begin{eqnarray*}
\overline{u}_2 (t, 2t +  \ln t +1) & =&  t^\varepsilon  \left( 1 - \frac{C}{t^{1/4}} \right) \frac{e^{-1}}{t^{3/2}}    e^{- \frac{\left(\frac{3}{2} \ln t + 1 \right)^2 }{4t}}   \\
& & \qquad  \times \left[ 1  - 2 \frac{ \frac{3}{2} \ln t +1}{t^{1-\varepsilon}} +  \frac{\left(\frac{3}{2} \ln t + 1 \right)^2}{t^{3/2 - \varepsilon} }\right]  \\
& = &    \frac{e^{-1} e^{h(t)} }{t^{3/2 - \varepsilon }}    e^{- \frac{3 \ln t}{4t }} e^{ - \frac{1}{ 4t}}   \\
&& \qquad + \frac{e^{-1}}{t^{3/2 - \varepsilon}}\left( 1 - \frac{C}{t^{1/4}} \right) \left[ - \frac{2}{t^{1-\varepsilon}} +  \frac{3 \ln t +1} {t^{3/2 - \varepsilon}} \right] e^{- \frac{\left(\frac{3}{2} \ln t + 1 \right)^2 }{4t}} \\
& < &  \frac{e^{-1} e^{h(t)}}{t^{3/2 - \varepsilon}} - \frac{e^{-1}}{ t^{5/2 - 2 \varepsilon}} .
\end{eqnarray*}
Thus $\overline{u}_2 (t , 2t + \ln t + 1) < \overline{u}_1 (t , 2t + \ln t +1)$ for any large enough~$t$.

We have finally obtained the following (generalized) supersolution, for all $t \geq T$ and~$x \in \R$:
$$\overline{u} (t,x) = \left\{\begin{array}{ll}
 \overline{u}_1 (t,x) & \mbox{ if } x \leq 2t + \ln t -1  , \vspace{3pt}\\
 \min \{ \overline{u}_1 (t,x)  ,\overline{u}_2 (t,x) \}  & \mbox{ if } 2t +  \ln t -1 < x < 2t + \ln t + 1, \vspace{3pt}\\
\overline{u}_2 (t,x) & \mbox{ if } x \geq 2t + \ln t  +1  .
\end{array}
\right.$$
Furthermore, it follows from the definition of $\overline{u}_1$ in \eqref{def:overlineu1} that there exists $X_2 >0$ such that $\overline{u} (T,x) \geq 1$ for any $x \leq - X_2$. Next, from the fact that $\overline{u}_1$ is also defined by~\eqref{eq:fast_super_ansatz1bis}, where $h(t)$ and $Z_\delta (t)$ are bounded, and $U_r$ is positive on the left of the point~$Z_0(t)$ with~\eqref{eq:ansatz1_decay2}, we infer that $\overline{u}_1$ is positive for $x \leq 2t + \ln t +1$. Since $\overline{u}_2$ is also positive by construction, we conclude that~$\overline{u}$ satisfies statement~$(ii)$ of Proposition~\ref{fast_supersol}.

Next, it follows from the definition of~$\overline{u}_1$ in~\eqref{def:overlineu1} and the fact that $h(t)$, $Z_\delta (t)$ are bounded, that for any $\lambda \in (0,1)$, there exists $X_1 >0$ such that
$$\forall x < 2t - r \ln t - X_1, \quad \overline{u} (t,x) >\lambda.$$
Lastly, by construction~$U_r (\cdot ; t)$ converges locally uniformly to (a shift of) $U_*$ as $t \to +\infty$; indeed recall~\eqref{revised_decay0}. Then, by~\eqref{eq:Z0_neg_deriv00} and~\eqref{eq:ansatz1_decay2}, we also have that~$\partial_z U_r (z;t) <0$ for any $0 \leq z \leq (r+2) \ln t$. From the definition of~$\overline{u}_1$ in~\eqref{eq:fast_super_ansatz1bis}, together with the fact that $\overline{u}_2 (t,x)$ goes to $0$ as $t \to +\infty$ uniformly with respect to $x \geq 2t + \ln t - 1$ thanks to \eqref{eq:lem:fast_super_ansatz21_upper}, we infer that for any $\lambda \in (0,1)$, there exists $X_1 >0$ such that
$$\forall x > 2t - r \ln t + X_1, \quad \overline{u} (t,x)< \lambda.$$
Therefore statement $(i)$ of Proposition~\ref{fast_supersol} holds true too. This concludes the proof of Proposition~\ref{fast_supersol}.

\subsection{The slow decay case: a supersolution}\label{sec:slow_super}

For any $f$ satisfying the monostable assumption~\eqref{eq:monostable}, it is possible to find a KPP type nonlinearity~$\tilde{f}$ such that $\tilde{f} ' (0) = f' (0)$ and $\tilde{f} (s) \leq f(s)$ for any $s \geq 0$. In particular, the lower estimate on the position of the level sets in statement~$(i)$ of Theorem~\ref{th:main} immediately follows from a comparison principle and the celebrated result of~\cite{Bramson,HNRR}. We refer to Section~\ref{sec:comparison} for the details.

Therefore, in the slow decay case it is enough to construct a supersolution. In this section we assume (up to some shift) that
\begin{equation}\label{ass:slowdecay_super}
U_* (z) = (z + A) e^{-z} + O ( e^{-(2-\eta) z} ), \quad U_* ' (z) = - (z+A -1)e^{-z}  + O (e^{-(2-\eta)z},
\end{equation} 
as $z \to +\infty$, where $\eta >0$ is arbitrarily small. Notice that, since we are shifting the front to make the constant $B$ in~\eqref{eq:asymp} equal to $1$, the constant $A \in \R$ is now also fixed. Then we prove the following:
\begin{prop}\label{slow_supersol}
There exist $T>0$ and $\overline{u} (t,x)$ a supersolution of~\eqref{eq:main} on $[T,+\infty) \times \R$ such that:
\begin{itemize}
\item[$(i)$] for any $\lambda \in (0,1)$, there exists $X_1 >0$ such that, for any $t$ large enough,
$$\forall x < 2 t - \frac{3}{2} \ln t - X_1, \quad \overline{u} (t,x) > \lambda,$$
$$\forall x > 2 t - \frac{3}{2} \ln t + X_1, \quad \overline{u} (t,x) < \lambda ;$$
\item[$(ii)$] the function $\overline{u}$ is positive on $[T,+\infty) \times \R$ and there exists $X_2 > 0 $ such that
$$\forall x \leq -X_2 , \quad \overline{u} (T,x) \geq 1.$$
\end{itemize}
\end{prop}
The general strategy remains the same as in Section~\ref{sec:fast_super}, and in particular the left supersolution is constructed in the same way. More precisely, we define (using the same notation for simplicity)
$$\overline{u}_1 (t,x) = \left\{
\begin{array}{ll}
U_* (x - 2t + \frac{3}{2} \ln t  ) + \frac{\gamma}{t} & \mbox{ if } t>0, \ x \leq 2t - \frac{3}{2} \ln t   -Z_\delta (t),  \vspace{3pt}\\
U_{\frac{3}{2}} (x - 2t + \frac{3}{2} \ln t   + Z_\delta (t) ; t) & \mbox{ if } t >0, \ x > 2t - \frac{3}{2} \ln t    -Z_\delta (t),
\end{array}
\right.
$$
where we recall that $\gamma >0$,  $Z_\delta (t)$ is such that $U_* (-Z_\delta (t)) = 1-\delta - \frac{\gamma}{t}$, $U_r (\cdot; t)$ solves the ODE~\eqref{newODE} and~\eqref{newODE_boundary}. Notice that, compared with Section~\ref{sec:fast_super}, here we choose $r= 3/2$, so that eventually we will recover the exact logarithmic drift, and also $h( t) \equiv 0$.
\begin{lem}\label{lem:slow_ansatz1}
There exist $\gamma >0$ and $T>0$ so that~$\overline{u}_1$ is a positive supersolution of~\eqref{eq:main} for $t \geq T$ and $x < 2t -\frac{3}{2} \ln t   -  Z_\delta (t) + Z_0 (t)$.

Moreover, the function $Z_0 (t)$ satisfies, for any $t \geq T$,
$$Z_0 (t) > \frac{7}{2} \ln t ,$$
and there holds
$$\sup_{ x \geq 2t + \ln t  + Z_\delta^\infty - Z_\delta (t) -2} \left|\overline{u}_1 (t, x ) - U_* \left(x  -2t + \frac{3}{2} \ln t  + Z_\delta (t) - Z_\delta^\infty \right) \right| = O \left( t^{- \frac{7}{2} (1-\eta)}  \right),$$
as $t \to +\infty$, with $\eta >0$ arbitrarily small and 
$$
Z_\delta (t) - Z_\delta^\infty = O \left( \frac{1}{t} \right).
$$
\end{lem}
\begin{proof}
The proof is the same as in Section~\ref{sec:fast_super} and thus we omit the details. We refer in particular to~\eqref{claim00_add_useful}, \eqref{eq:ansatz1_decay1} and~\eqref{eq:ansatz1_decay2}.\end{proof}

Let us then focus the discussion on the right supersolution. We first place ourselves in the moving frame with a logarithmic drift, more precisely in the moving frame around $2t - \left( \frac{3}{2} + \varepsilon \right) \ln t$ with $\varepsilon >0$ to be made arbitrarily small. Of course this coincides with the expected drift when $\varepsilon =0$, but we need a slight gap with the position of the level sets of~$\overline{u}_1$ in order to later match our left and right supersolutions. We obtain the linearized equation
$$\partial_t u = \partial_{zz}  u + \left( 2 - \frac{3 + 2 \varepsilon}{2t} \right)  \partial_z u + u.$$
Letting $u= e^{-z} v$ and switching to the self-similar variables $\tau = \ln t$ and $y = \frac{z}{\sqrt{t}}$, we get
$$\partial_\tau v = \partial_{yy} v +  \frac{y}{2} \partial_y v + \frac{3 +2 \varepsilon}{2} v - \frac{3 + 2 \varepsilon}{2} e^{- \frac{\tau}{2}} \partial_y v .$$
Then we let
$$v = w \times e^{-\frac{y^2}{8}} e^{ \frac{ \tau}{2}} e^{\varepsilon \tau},$$
so that
\begin{equation}\label{eq:w}
\partial_\tau w = \partial_{yy} w - \left(\frac{y^2}{16} -\frac{3}{4} \right)w  - \frac{3+ 2 \varepsilon}{2} e^{-\frac{\tau}{2}} \left(\partial_y w - \frac{y}{4} w\right).
\end{equation}
Let us now define $\tilde{w}$ the solution of \eqref{eq:w} on the right half-line $\{ y > 0 \}$, together with the Dirichlet boundary condition
$$ \tilde{w} (\tau,0)=0 , $$
for any $ \tau >0$, and the initial data
$$ \tilde{w} (0, y)= y e^{-\frac{y^2}{8}},$$
for $y >0$. We point out that this initial condition is none other than the first Dirichlet eigenfunction of the autonomous part, which is the same as in Sections~\ref{sec:fast_sub} and~\ref{sec:fast_super}; see~\eqref{eq:nonautonomous}. It is a byproduct of the proof of~\cite[Lemma~2.2]{HNRR} that there exists some $W_1 >0$, and for any bounded interval $[0,L]$ there exists some $K(L) >0$, such that
\begin{equation}\label{eq:slow_westim_rev} \left| \tilde{w} (\tau, y) -  W_1 y e^{- \frac{y^2}{8}} \right|  \leq K (L) y e^{-\frac{\tau}{2}},
\end{equation}
for any $\tau \geq 1$ and $y \in [0,L]$. We also refer to Appendix~\ref{appen} for a similar proof of the corresponding result in the case of a Neumann boundary condition. We further point out that $ e^{\tau} e^{y^2/8}$ is a supersolution of~\eqref{eq:w}, thus by the comparison principle, we get that
\begin{equation}\label{eq:slow_westim_rev2}
\tilde{w} (\tau , y) \leq e^\tau e^{y^2 /8},
\end{equation}
for any $\tau \geq 0$ and $y \geq 0$.

Now let us denote
$$\overline{w} (\tau, y ) := \frac{\beta}{W_1}  \tilde{w} (\tau,y),$$
with 
\begin{equation}\label{eq:beta_last}
\frac{1}{1 + \frac{2 \varepsilon}{5}} <   \beta < \frac{1}{1 + \frac{\varepsilon}{3}}.
\end{equation}
Notice that $\overline{w}$ still solves~\eqref{eq:w}. The choice of the factor~$\beta$ slightly less than~$1$ will ensure that the spatial derivative of the right supersolution is slightly less than the spatial derivative of the left supersolution~$\overline{u}_1$. In particular it will be crucial when merging the two together into a generalized supersolution.
Equipped with $\overline{w}$ and going back to the original problem in the moving frame of the expected drift, we find a function $\tilde{u} (t,z)$ such that
$$\partial_t \tilde{u} - \partial_{zz} \tilde{u} - \left( 2 -  \frac{3+ 2\varepsilon}{2t} \right) \partial_z \tilde{u} - \tilde{u} \geq 0.$$
More precisely we can now prove the following lemma:
\begin{lem}\label{lem:slow_ansatz2}
For any $\varepsilon \in \left( 0 , \frac{1}{4}  \right)$, there exist $C>0$ and $T>0$ such that the function
$$\overline{u}_2 (t,x) := \left( 1 - \frac{C}{t^{1/4}} \right) \tilde{u} \left(t,x -2t + \frac{3 + 2 \varepsilon}{2} \ln t \right),$$
is a supersolution of \eqref{eq:main} in the subdomain $t \geq T$ and $x \geq 2t $, where~$\tilde{u}$ is a positive function which satisfies that for any $L>0$, there exists $K(L)>0$ such that
\begin{equation}\label{slow_ansatz_estim}
\left| \tilde{u} (t,z) - \beta t^\varepsilon  z e^{-z} e^{-\frac{z^2}{4t}}  \right| \leq \textcolor{blue}{\frac{ \beta K(L)}{W_1}} \frac{z}{t^{1/2 - \varepsilon}}  e^{-z},
\end{equation}
for any $t \geq e^1$ and $z \in [0,L \sqrt{t}]$, as well as
\begin{equation}\label{slow_ansatz_upperbound}
\tilde{u} (t,z) \leq \frac{\beta}{W_1} t^{\varepsilon+ 3/2} e^{-z}, 
\end{equation}
for any $t \geq 1$ and $z \geq 0$.
\end{lem}
\begin{proof}
We proceed as in the proof of Lemma~\ref{lem:fast_super_ansatz22}. First, as we outlined above, we introduce
$$\tilde{u} (t,z) = e^{-z} e^{-\frac{y^2}{8}} e^{\frac{\tau}{2}} e^{\varepsilon \tau} \times \overline{w} (\tau,y),$$
where $\tau = \ln t$, $y = \frac{z}{\sqrt{t}}$, and $\overline{w} (\tau,y) = \frac{\beta}{W_1} \tilde{w} (\tau,y)$ with~$\tilde{w}$ the (positive) solution of~\eqref{eq:w} with Dirichlet boundary condition $\tilde{w} (\cdot,0) \equiv 0$. Then \eqref{slow_ansatz_estim} and \eqref{slow_ansatz_upperbound} immediately follow respectively from~\eqref{eq:slow_westim_rev} and~\eqref{eq:slow_westim_rev2}.

Next, by construction we have that~$\tilde{u}$ satisfies
$$\partial_t \tilde{u} - \partial_{zz} \tilde{u}  - \left( 2 - \frac{3+ 2\varepsilon}{2t } \right) \partial_z \tilde{u}  - \tilde{u}  \geq 0,$$
for any $t \geq 1$ and $z \geq 0$. Letting $K = \| f'' \|_\infty$, we get
\begin{eqnarray*}
 \partial_t \overline{u}_2 - \partial_{xx} \overline{u}_2 -  f(\overline{u}_2) & \geq & \tilde{u} \left( \frac{C}{4 t^{5/4}} - K \tilde{u}\right),
\end{eqnarray*}
where $\tilde{u}$ is evaluated at $(t,x -2t + \frac{3 + 2 \varepsilon }{2 } \ln t)$.

Now, for all $z =x-2t + \frac{3+2\varepsilon}{2} \ln t \in \left[ \frac{3}{2} \ln t, \sqrt{t} \right]$, we can use~\eqref{slow_ansatz_estim} to get that 
\begin{eqnarray*}
\tilde{u} (t,z) & \leq & \beta t^\varepsilon z e^{-z} e^{-\frac{z^2}{4t}} + \frac{\beta K(1)}{W_1} \frac{z}{t^{1/2 -\varepsilon}} e^{-z} \\
 & \leq & \beta \left( t^\varepsilon + \frac{K(1)}{W_1 t^{1/2 - \varepsilon}} \right) z e^{-z} \\
& \leq & \beta \left( t^\varepsilon + \frac{K(1)}{W_1 t^{1/2 - \varepsilon}} \right) \frac{3}{2  t^{3/2}}\ln t  = o \left( t^{-5/4} \right),
\end{eqnarray*}
as $t \to +\infty$. A similar estimate holds for $x - 2t + \frac{3+2\varepsilon}{2} \ln t \geq  \sqrt{t}$ thanks to \eqref{slow_ansatz_upperbound}. It follows that, up to increasing $T$,
$$\partial_t \overline{u}_2 - \partial_{xx} \overline{u}_2 - f (\overline{u}_2) \geq 0 ,$$
on the subdomain $t \geq T$ and $x \geq 2t$. Lemma~\ref{lem:slow_ansatz2} is proved. \end{proof}

Let us now try to match the supersolutions $\overline{u}_1$ and $\overline{u}_2$ in order to conclude the proof of Proposition~\ref{slow_supersol}. We first compute, using \eqref{ass:slowdecay_super} and Lemma~\ref{lem:slow_ansatz1}, that
\begin{eqnarray*}
\overline{u}_1 (t, 2t + \ln t) &=& U_* \left( \frac{5}{2} \ln t + Z_\delta (t) -Z_\delta^\infty  \right) +   O \left( t^{- \frac{7}{2}(1-\eta) } \right)\\
& = &   \frac{ \frac{5}{2}  \ln t   + A}{t^{5/2}}   + O \left( t^{-\frac{7}{2}(1-\eta) }   \right) \\
& = & \frac{ 5 \ln t}{2 t^{5/2 }}   + O \left( \frac{1}{t^{5/2}}\right),
\end{eqnarray*}
as $t \to +\infty$, where $\eta >0$ can be made arbitrarily small. 
On the other hand, by~\eqref{slow_ansatz_estim} and the definition of $\overline{u}_2$ in Lemma~\ref{lem:slow_ansatz2},
\begin{eqnarray*}
\overline{u}_2 (t,2t + \ln t) &= &  \left( 1 - \frac{C}{t^{1/4}} \right) \tilde{u} \Big( t,\frac{5 + 2 \varepsilon}{2} \ln t \Big) \\
& = & \left( 1 - \frac{C}{t^{1/4}} \right)\times  \beta  t^\varepsilon \frac{5 + 2\varepsilon}{2}  \ln t  \times   t^{-\varepsilon - 5/2}e^{-\frac{(5+ 2\varepsilon)^2 (\ln t)^2}{16t}} + O \left( \frac{\ln t}{t^3} \right)\\
& = & \frac{5 + 2 \varepsilon}{2} \beta \frac{\ln t}{t^{5/2}} + o \left( \frac{\ln t}{t^{5/2}} \right),
\end{eqnarray*}
as $t \to +\infty$. From the left inequality in~\eqref{eq:beta_last}, we have that $\frac{5 +2 \varepsilon}{2} \beta > \frac{5}{2}$,
and then we get that $\overline{u}_2 (t,2t + \ln t ) > \overline{u}_1 (t,2t + \ln t)$ for any $t$ large enough.

Next, we compute 
\begin{eqnarray*}
\overline{u}_1 \left( t, 2t + \frac{3}{2} \ln t \right) &=& U_* \left(3 \ln t + Z_\delta (t) - Z_\delta^\infty \right) +   O \left(  t^{-\frac{7}{2} (1-\eta)}  \right)\\
& = &  \frac{ 3  \ln t   }{t^{3} } + O \left( \frac{1}{t^3}  \right) ,
\end{eqnarray*}
as $t \to +\infty$. On the other hand,
\begin{eqnarray*}
\overline{u}_2 \left( t,2t + \frac{3}{2} \ln t \right) &= &  \left( 1 - \frac{C}{t^{1/4}} \right) \tilde{u} ( t, (3+ \varepsilon) \ln t) \\
& = & \frac{(3+ \varepsilon) \beta  \ln t }{  t^{3 }} + o \left( \frac{\ln t}{t^{3}} \right) ,
\end{eqnarray*}
as $t \to +\infty$. Using~\eqref{eq:beta_last} again, we find that $(3+\varepsilon) \beta < 3 $ and $\overline{u}_2 (t, 2t + \frac{3}{2} \ln t) < \overline{u}_1 (t,2t + \frac{3}{2} \ln t)$ for any $t$ large enough.

We conclude that
$$\overline{u} (t,x) = \left\{\begin{array}{ll}
 \overline{u}_1 (t,x) & \mbox{ if } x \leq 2t  + \ln t  , \vspace{3pt}\\
 \min \{ \overline{u}_1 (t,x)  ,\overline{u}_2 (t,x) \}  & \mbox{ if } 2t + \ln t < x < 2t + \frac{3}{2} \ln t , \vspace{3pt}\\
\overline{u}_2 (t,x) & \mbox{ if } x \geq  2t + \frac{3}{2} \ln t  ,
\end{array}
\right.$$
defines a positive supersolution of \eqref{eq:main} for $t \geq T$ and $x \in \R$. Moreover, the function $\overline{u}$ satisfies the required properties $(i)$ and $(ii)$ of Proposition~\ref{slow_supersol}. Since the argument is the same as in the proof of Proposition~\ref{fast_supersol}, we omit the details. This ends the proof of Proposition~\ref{slow_supersol}.

\section{Concluding the proof of Theorem~\ref{th:main}}\label{sec:comparison}

We are now in a position to end the proof of Theorem~\ref{th:main}. Recall that we assume, without loss of generality and up to some rescaling, that $f'(0)= 1$. 

\subsection{The slow decay case}\label{sec:slowcomp_rev}

We first check statement $(i)$ of Theorem~\ref{th:main}, which is the slow decay case. First notice that there exists some $K>0$ large enough such that
$$\forall s \geq 0, \quad   f (s) \geq \tilde{f} (s) := s - Ks^2.$$
In particular, applying the comparison principle, we get that $u (t,x)$ the solution of \eqref{eq:main}-\eqref{eq:ini0} satisfies
$$u \geq \tilde{u},$$
where $\tilde{u}$ solves
$$\partial_t \tilde{u} = \partial_{xx} \tilde{u} + \tilde{f} (\tilde{u}),$$
with $\tilde{u} (t=0, \cdot) \equiv u_0$. Since $\tilde{f}$ is of the KPP type, we can apply the well-known result from~\cite{Bramson,HNRR} and, together with a comparison principle, we conclude that for any $\lambda \in (0, 1/K)$,
\begin{equation}\label{eq:slow_lower0}
\exists X(\lambda) >0, \quad E_\lambda (t) \subset \left[  2 t - \frac{3}{2} \ln t - X (\lambda) , +\infty \right),
\end{equation}
where $E_\lambda (t)$ is the $\lambda$-level set of $u(t,\cdot)$. 

Now consider $\lambda \in [ 1/K, 1)$ and let us again check~\eqref{eq:slow_lower0}. We proceed by contradiction and assume that there exist sequences $t_n \to +\infty$ and $x_n \in \mathbb{R}$ such that
$$u (t_n, x_n ) = \lambda,$$
and
$$x_n - 2t + \frac{3}{2} \ln t \to - \infty.$$
By standard parabolic estimates, we can extract a subsequence such that $u(t+t_n, x+x_n)$ converges to some entire in time solution~$u_\infty$ of \eqref{eq:main}. Moreover $u_\infty (0,0) = \lambda$ and, from the fact that~\eqref{eq:slow_lower0} holds true for any $ \lambda \in (0, 1/ K)$, we know that
$$ \frac{1}{K} \leq u_\infty \leq 1 ,$$
in $\mathbb{R}^2$. By comparison with the solutions of the ODE $\partial_t v = f(v)$ with $v(0) = \frac{1}{K}$, one may conclude that $u_\infty \equiv 1$, a contradiction. The lower estimate on the position of the $\lambda$-level set is proved.\\

Let us now turn to the upper estimate, and let $\overline{u}$ be the supersolution provided by Proposition~\ref{slow_supersol}. We claim that there exists $X>0$ such that
$$
u (0, x ) \leq \overline{u} (T,x - X ),
$$
for all $x \in \mathbb{R}$. Indeed, it is sufficient to take $X=  X_0 + X_2$, where $X_0$ comes from~\eqref{eq:ini} and $X_2$ from statement~$(ii)$ of Proposition~\ref{slow_supersol}. More precisely, for $x \geq X_0$ we have $u_0 (x) =0 \leq \overline{u} (T,x-X)$, and for $x \leq X_0$ then $x  - X\leq - X_2$ and $\overline{u} (T,x-X)\geq 1 \geq u_0 (x)$.

Thus, by the parabolic comparison principle, for any $t \geq 0$ and $x \in \mathbb{R}$ we have
$$u (t, x) \leq \overline{u} (t + T,x - X).$$
It now follows from statement $(i)$ of Proposition~\ref{slow_supersol} that, for any $\lambda \in (0,1)$, there exists~$X_1$ such that
$$E_\lambda (t) \subset \left( -\infty,  2 (t+T) - \frac{3}{2} \ln (t+T) + X_1  + X\right].$$
This concludes the proof of statement~$(i)$ of Theorem~\ref{th:main}.

\subsection{The fast decay case}

We now turn to statement~$(ii)$ of Theorem~\ref{th:main}. We let $\varepsilon >0$ be arbitrarily small and $\underline{u}$, $\overline{u}$ be the sub and supersolution from Propositions~\ref{fast_subsol} and~\ref{fast_supersol}, respectively with $r \in \left ( \frac{1}{2} , \frac{1+\varepsilon}{2}\right)$ and $r \in \left( \frac{1- \varepsilon}{2}, \frac{1}{2} \right)$.

Similarly as before, it is enough to prove that there exists $X>0$ large enough such that
\begin{equation}\label{eq:comp_final_claim1}
 u (0,x)  \leq \overline{u} (T , x - X),
\end{equation}
as well as $t_0 >0$ such that
\begin{equation}\label{eq:comp_final_claim2}
\underline{u} (T, x  +  X) \leq u (t_0, x ), 
\end{equation}
for all $x \in \mathbb{R}$. Regarding~\eqref{eq:comp_final_claim1}, it is again enough to choose $X\geq X_0 + X_2$, where $X_0$ comes from~\eqref{eq:ini} and $X_2$ from statement~$(ii)$ of Proposition~\ref{fast_supersol}. 

Let us then check~\eqref{eq:comp_final_claim2}. First, recall from~\eqref{eq:ini} that 
$$\liminf_{x \to -\infty} u_0 (x) >0.$$
It follows that
$$\liminf_{x \to - \infty} u(t,x) \geq v(t),$$
for any $t>0$, where $v$ solves $v' = f(v)$ together with $v(0) = \liminf_{x \to -\infty} u_0 (x) >0$, hence 
$$\lim_{t \to +\infty} \liminf_{x \to -\infty} u(t,x) = 1.$$
In particular, by statement~$(ii)$ of Proposition~\ref{fast_subsol} we can choose $t_0 > T$ such that
$$\liminf_{x \to -\infty} u(t_0,x) > \sup_{x \in \mathbb{R}} \underline{u} (T,x),$$
and up to increasing $X_0$ we can assume that
\begin{equation}\label{last_comp1}
\inf_{x \leq -X_0 +1 } u (t_0, x) > \sup_{x \in \mathbb{R}} \underline{u} (T,x).
\end{equation}
Furthermore, there exists $K>0$ large enough such that
$$\partial_t u \geq \partial_{xx} u - Ku,$$
and by the comparison principle,
$$u (t,x) \geq \frac{e^{-Kt}}{\sqrt{4 \pi t}} \int_{\mathbb{R}} u_0 (y)e^{-\frac{(x-y)^2}{4t} }  dy$$
for any $t >0$ and $x \in \mathbb{R}$. Next, using again \eqref{eq:ini} and the fact that
$$\int_z^{+\infty} e^{-s^2} ds \sim \frac{e^{-z^2}}{2z}$$
as $z \to +\infty$, we get that
$$u(t_0,x) \geq \frac{e^{-Kt_0}}{\sqrt{4 \pi t_0}}  \int_{-\infty}^{-X_0} \left( \inf_{z \leq - X_0}u_0 (z ) \right) e^{-\frac{(x-y)^2}{4t_0}} dy \geq \frac{\delta}{x+ X_0} e^{-\frac{(x +X_0)^2}{4t_0}},$$ 
for some $\delta >0$ and any $x \geq -X_0 +1$. On the other hand, from statement~$(ii)$ of Proposition~\ref{fast_subsol}, we have that
$$\exists X_2 >0, \quad \limsup_{x \to +\infty} \frac{(x- X_2) \underline{u} (T,x)}{e^{-(x- X_2)^2 /4T}} < \delta .$$
Thus, for $X \geq X_0 + X_2$ large enough, we have that
$$\underline{u} (T,x+X) \leq \frac{\delta}{x+ X-X_2} e^{-\frac{(x + X- X_2)^2}{4T}}  \leq u (t_0,x),$$
for any $x \geq - X_0 +1$; notice that we used the fact that $t_0 >T$. Together with \eqref{last_comp1}, we get \eqref{eq:comp_final_claim1}.

Finally, with \eqref{eq:comp_final_claim1} and \eqref{eq:comp_final_claim2} in hand, the wanted estimates on $E_\lambda (t)$ follow by the parabolic comparison principle. We omit the details since the argument proceeds almost exactly as in Section~\ref{sec:slowcomp_rev}. Theorem~\ref{th:main} is proved.

\appendix

\section{Appendix: proof of \eqref{eq:appen}}\label{appen}

Here we want to prove the following result.
\begin{prop}\label{prop:app}
Fix a nonnegative and nontrivial function $\tilde{w}_0$ such that
$$\tilde{w}_0 \in L^2 ((0,+\infty)), \quad  [ y \mapsto y \tilde{w}_0 (y) ] \in L^2 ((0,+\infty)).$$
Let also $\tilde{w}$ solve
$$\partial_\tau \tilde{w} -\partial_{yy} \tilde{w} + \left( \frac{y^2}{16} - \frac{3}{4} \right) \tilde{w} + \frac{1}{2} e^{-\frac{\tau }{2}} \left( \partial_y \tilde{w} - \frac{y}{4} \tilde{w} \right) = 0,$$
for $\tau > 0 $ and $y >0$, together with the Neumann boundary condition
$$ \partial_y \tilde{w} (\tau, 0) = 0,$$
for $\tau > 0$, and the initial data
$$\tilde{w} (0, y) = \tilde{w}_0 ,$$
for $y >0$.

Then there exists $W_1 >0$, and for any bounded interval $[0,L]$ there exists $K (L) >0$, such that
$$ \left| \tilde{w}  (\tau,y) -  W_1 e^{-\frac{y^2}{8}}e^{\frac{\tau}{2}}  \right| \leq K(L)   ,$$
for any $\tau \geq 1$, and $y \in [0,L]$.
\end{prop}
\begin{proof}
First we write
$$\phi (\tau, y)= e^{-\frac{\tau}{2}}  \tilde{w} (\tau,y),$$
which solves
$$\partial_\tau \phi = \partial_{yy}  \phi -  \left( \frac{y^2}{16} - \frac{1}{4} \right) \phi - \frac{1}{2} e^{-\frac{\tau }{2}} \left( \partial_y \phi - \frac{y}{4} \phi \right)  .$$
Similarly as in \cite{HNRR}, we define
$$M_1 \phi := - \partial_{yy} \phi + \left(\frac{y^2}{16} - \frac{1}{4}  \right) \phi,$$
on the space of functions $\phi \in H^1 ((0,+\infty))$ with $y\phi \in L^2 ((0,+\infty))$, and we rewrite the equation as
\begin{equation}\label{eq:app_phi}
\partial_\tau \phi + M_1 \phi = -\frac{1}{2} e^{-\frac{\tau }{2}} \left(\partial_y \phi - \frac{y}{4} \phi \right).
\end{equation}
Actually, accounting for the Neumann boundary condition, we should understand $M_1$ in the appropriate weak formulation, that is
$$\langle M_1  \phi , \psi \rangle  = \int_0^{+\infty} \left( \partial_y \phi  \, \partial_y \psi + \left(\frac{y^2}{16} - \frac{1}{4 } \right) \phi  \psi \right) dy, $$
for any $\psi \in H^1 ((0,+\infty))$ with $y \psi \in L^2 ((0,+\infty))$. Hereafter $\langle \cdot, \cdot \rangle$ denotes the usual $L^2((0,+\infty))$ scalar product. The null space of $M_1$ is generated by the first unit eigenfunction
$$e_0 (y) := \frac{1}{(2\pi)^{1/4}} e^{-\frac{y^2}{8}},$$
while the second eigenvalue is 1, which is associated with the eigenfunction $(e^{-\frac{y^2}{4}})'' e^{\frac{y^2}{8}}$, i.e. $\left( \frac{y^2}{4} - \frac{1}{2} \right) e^{-\frac{y^2}{8}}$. Higher order eigenfunctions are also given by Hermite polynomials and they form an orthonormal basis of $L^2 ((0,+\infty))$. It follows that the quadratic form
$$Q_1 (\psi) := \langle M_1 \psi, \psi \rangle =  \int_0^{+\infty} \left( (\partial_y \psi)^2 + \left(\frac{y^2}{16} - \frac{1}{4 }\right) \psi^2\right) dy, $$
satisfies
\begin{equation}\label{eq:app_Qperp}
 Q_1 (\psi) \geq \| \psi \|_{L^2}^2 \qquad \text{in $e_0^{\perp}$}.
\end{equation}
We point out that, for conciseness, we choose to denote $\| \cdot \|_{L^2}$ instead  of $\| \cdot\|_{L^2 ((0,+\infty))}$. Multiplying~\eqref{eq:app_phi} by $\phi$ and integrating by parts, we find
\begin{eqnarray*}
\partial_\tau \| \phi (\tau, \cdot) \|_{L^2}^2 + 2 Q_1 (\phi (\tau,\cdot)) &=&  e^{-\frac{\tau }{2}} \int_0^{+\infty} \left( \frac{y}{4}\phi (\tau,y)^2  -\phi (\tau,y) \partial_y \phi (\tau,y) \right) dy. 
\end{eqnarray*}
Using that
$$\frac{y}{4} \phi (\tau,y)^2 = \left( \frac{y}{2\sqrt{2}} \phi (\tau,y ) \right) \left( \frac{\phi (\tau,y)}{\sqrt{2}} \right) \leq \frac{y^2}{16} \phi (\tau,y)^2 + \frac{ \phi (\tau,y)^2}{4},$$
and
$$- \phi (\tau,y) \partial_y \phi (\tau,y) = \left( - \frac{\phi (\tau,y) }{\sqrt{2}} \right) \left( \sqrt{2} \partial_y \phi (\tau,y) \right) \leq \frac{\phi (\tau,y)^2}{4} + (\partial_y \phi (\tau,y))^2,$$
we get
\begin{eqnarray*}
\partial_\tau \| \phi (\tau, \cdot) \|_{L^2}^2 + 2 Q_1 (\phi (\tau,\cdot)) 
& \leq & e^{-\frac{\tau }{2}} \int_0^{+\infty} \bigg( \frac{y^2}{16}\phi (\tau,y)^2 + \frac{1}{4} \phi(\tau,y)^2   \\
& &  \hspace{1cm} + \frac{1}{4} \phi (\tau,y)^2 + (\partial_y \phi (\tau,y))^2 \bigg) dy \\
& \leq & e^{-\frac{\tau }{2}} \left( Q_1 (\phi (\tau,\cdot))  +  \|  \phi (\tau,\cdot) \|_{L^2}^2 \right).
\end{eqnarray*}
Thus
\begin{eqnarray*}
\partial_\tau \| \phi (\tau, \cdot ) \|_{L^2}^2 & \leq  &  e^{-\frac{\tau }{2}}  \| \phi (\tau,\cdot)\|_{L^2}^2 ,
\end{eqnarray*}
and $\| \phi (\tau,\cdot)\|_{L^2}$ is bounded uniformly with respect to $\tau \geq 0$. By parabolic regularity, so is $\phi (\tau,0)$.

Now set
$$\phi (\tau,y) = \langle e_0 , \phi  (\tau,\cdot) \rangle e_0 (y) + \tilde{\phi} (\tau,y), \quad \text{ with } \tilde{\phi} (\tau,\cdot) \in e_0^{\perp}.$$
Then $\phi_1 (\tau) := \langle e_0 , \phi (\tau,\cdot) \rangle$ satisfies 
\begin{eqnarray*}
\phi'_1 (\tau) & = &  \langle e_0 , \partial_\tau \phi \rangle \\
& = & \langle   e_0 , -M_1 \phi \rangle -\frac{1}{2} e^{-\frac{\tau}{2}} \langle e_0 , \partial_y \phi - \frac{y}{4} \phi \rangle \\
& = & -\langle M_1  e_0 , \phi \rangle -\frac{1}{2} e^{-\frac{\tau}{2}} \langle e_0 , \partial_y \phi - \frac{y}{4} \phi \rangle \\
& = & -\frac{1}{2} e^{-\frac{\tau}{2}} \langle e_0 , \partial_y \phi - \frac{y}{4} \phi \rangle \\
& =  & \frac{1}{2}e^{-\frac{\tau}{2}} \left( \langle \partial_y e_0 + \frac{y}{4} e_0 , \phi (\tau, \cdot) \rangle + e_0 (0) \phi (\tau,0) \right) \\
& = & \frac{e^{-\frac{\tau }{2}}}{2(2\pi)^{1/4}} \phi (\tau,0).
\end{eqnarray*}
Due to the time-uniform bound on~$\phi (\tau,0)$, we get that $\phi'_1$ is integrable. It is also positive and therefore there exists 
\begin{equation}\label{app:lim0}
\phi_1^\infty := \lim_{\tau \to +\infty} \phi_1 (\tau) >0,
\end{equation}
(notice indeed that $\phi_1 (0)$ is positive by the fact that $\tilde{w}_0$ is nonnegative and nontrivial). Moreover, there also exists $K_1 >0 $ such that
\begin{equation}\label{app:first_estim}
| \phi_1 (\tau) - \phi_1^\infty | = \int_\tau^{+\infty} \frac{e^{-\frac{s}{2}}}{2 (2 \pi)^{1/4}} \phi (s,0) ds <  K_1 e^{-\frac{\tau}{2}},
\end{equation}
for all $\tau \geq 0$.

On the other hand, $\tilde{\phi} (\tau,y)$ satisfies, together with the Neumann boundary condition,
\begin{eqnarray*}\partial_\tau \tilde{\phi} + M_1 \tilde{\phi}
& = & \partial_\tau \phi + M_1 \phi - \phi'_1 (\tau) e_0 \\
 & =  & - \frac{1}{2} e^{-\frac{\tau}{2}} \left(  \partial_y \tilde{\phi} - \frac{y}{4} \tilde{\phi} + \langle e_0, \phi (\tau,\cdot) \rangle \left( \partial_y e_0 - \frac{y}{4} e_0 \right) + \frac{1}{(2\pi)^{1/4}} \phi (\tau,0 ) e_0  \right).
\end{eqnarray*}
We multiply by $\tilde{\phi}$ and integrate to get:
\begin{eqnarray*}
\partial_\tau \| \tilde{\phi} (\tau, \cdot) \|_{L^2}^2 + 2 Q_1 (\tilde{\phi} (\tau,\cdot)) 
& = &  -  e^{-\frac{\tau }{2}}  \Big[  \langle \tilde{\phi} (\tau,\cdot), \partial_y \tilde{\phi} (\tau,\cdot) \rangle - \langle \frac{y}{4} \tilde{\phi} (\tau,\cdot) , \tilde{\phi} (\tau,\cdot) \rangle \\
& & \ +  \langle e_0, \phi (\tau,\cdot) \rangle \langle  \partial_y e_0 - \frac{y}{4} e_0 , \tilde{\phi}(\tau,\cdot) \rangle  \\
& & + \frac{1}{ (2\pi)^{1/4}} \phi (\tau,0) \langle e_0 , \tilde{\phi} (\tau,\cdot) \rangle \Big].
\end{eqnarray*}
Similarly as before, we have
\begin{eqnarray*}
- \langle \tilde{\phi} (\tau,\cdot), \partial_y \tilde{\phi} (\tau, \cdot) \rangle + \langle \frac{y}{4} \tilde{\phi} (\tau,\cdot), \tilde{\phi} (\tau,\cdot)\rangle  & = & \int_0^{+\infty} \left( \frac{y}{4} \tilde{\phi} (\tau,y)^2 - \tilde{\phi} (\tau, y) \partial_y \tilde{\phi} (\tau,y)\right) dy\\
& \leq & \int_0^{+\infty} \bigg( \frac{y^2}{16} \tilde{\phi} (\tau,y)^2 + \frac{1}{4} \tilde{\phi} (\tau,y)^2 \\
& & \hspace{1cm} + \frac{1}{4} \tilde{\phi} (\tau, y)^2 + ( \partial_y \tilde{\phi} (\tau,y))^2 \bigg) dy\\
& \leq & Q_1 (\tilde{\phi} (\tau,\cdot)) + \| \tilde{\phi} (\tau,\cdot) \|_{L^2}^2.
\end{eqnarray*}
Thus, and using also the fact that $\phi (\tau, 0)$ and the $L^2$-norm of $\phi (\tau,\cdot)$ are bounded uniformly with respect to time,
\begin{eqnarray*}
\partial_\tau \| \tilde{\phi} (\tau, \cdot) \|_{L^2}^2 + 2 Q_1 (\tilde{\phi} (\tau,\cdot)) 
&  \leq  &  e^{-\frac{\tau}{2}} \left(  Q_1 (\tilde{\phi} (\tau,\cdot)) + \| \tilde{\phi} (\tau,\cdot) \|_{L^2}^2   +  C  \| \tilde{\phi} (\tau,\cdot) \|_{L^2} \right),
\end{eqnarray*}
for some constant $C>0$. Then, recalling that $\tilde{\phi} \in e_0^\perp$ and \eqref{eq:app_Qperp}, we get
\begin{eqnarray*}
\partial_\tau \| \tilde{\phi} (\tau,\cdot) \|_{L^2}^2  & \leq&  ( e^{-\frac{\tau}{2}} -2) Q_1 (\tilde{\phi} (\tau,\cdot)) +   e^{-\frac{\tau}{2}} \| \tilde{\phi} (\tau,\cdot) \|_{L^2}^2  +  C  e^{-\frac{\tau}{2}} \| \tilde{\phi} (\tau,\cdot) \|_{L^2} \\
 & \leq &  ( 2 e^{-\frac{\tau}{2}} -2) \| \tilde{\phi} (\tau,\cdot) \|_{L^2}^2  +  C  e^{-\frac{\tau}{2}} \| \tilde{\phi} (\tau,\cdot) \|_{L^2} .
\end{eqnarray*}
Notice that $\psi (\tau) := C ' e^{-\tau}$ satisfies that
$$\partial_\tau \psi \geq (2 e^{-\frac{\tau}{2}} -2) \psi + C e^{-\frac{\tau}{2}} \sqrt{\psi},$$
for $C' >0$ large enough. It follows by a comparison argument that, for any $\tau \geq 0$,
$$\| \tilde{\phi} (\tau,0)\|_{L^2} \leq  \sqrt{C'}  e^{-\frac{\tau}{2}} ,$$
for some $C' >0$. Recall that by $L^2$-norm we refer here to the $L^2 ((0,+\infty))$-norm. By parabolic regularity, we get on any bounded interval $[0,L]$ and for any $\tau \geq 1$ that
$$\| \tilde{\phi} (\tau,\cdot) \|_{L^\infty ([0,L]) } \leq K_2 (L) e^{-\frac{\tau}{2}},$$
for some $K_2 (L) >0$. Thus, for any $\tau \geq 0$ and $y \in [0,L]$,
$$| \phi (\tau,y) - \phi_1^\infty e_0 (y) | \leq | \phi_1 (\tau) - \phi_1^\infty | e_0 (y) + K_2 (L) e^{-\frac{\tau}{2}},$$
where we recall from~\eqref{app:lim0} that $\phi_1^\infty = \lim_{\tau \to +\infty} \langle e_0 , \phi (\tau,\cdot) \rangle$. Using also~\eqref{app:first_estim}, the fact that $e_0 (y) \leq \frac{1}{(2 \pi)^{1/4}}$ and letting 
$$K (L) = \frac{K_1}{(2\pi)^{1/4}} + K_2 (L),$$
 we get
$$| \phi (\tau,y) - \phi_1^\infty e_0 (y) | \leq  K (L) e^{-\frac{\tau}{2}}.$$
Finally, rewriting the previous inequality in terms of $\tilde{w} (\tau,y) = e^{\frac{\tau}{2}} \phi (\tau,y)$, replacing~$e_0$ by its definition $\frac{1}{(2\pi)^{1/4}} e^{-\frac{y^2}{8}}$ and denoting $W_1 = \frac{ \phi_1^\infty}{(2 \pi)^{1/4}}$ , we obtain the wanted conclusion. Proposition~\ref{prop:app} is proved.
\end{proof}

\section*{Acknowledgements}

The author is grateful to Maolin Zhou for raising the issue considered in this article, and also for fruitful discussions throughout completion of this manuscript. Part of this work was carried out under the framework of ANR-21-CE40-0008 ``Indyana''.

\end{document}